\documentclass[a4paper,reqno,11pt]{amsart}
\usepackage{amsfonts}
\usepackage{amsmath}
\usepackage{amsthm, amsmath, amssymb, enumerate}
\usepackage{amssymb, enumitem}
\usepackage{bbm}
\usepackage{amscd}
\usepackage[all]{xy}
\usepackage[left=2cm,right=2cm,bottom=3cm,top=3cm]{geometry}
\usepackage[hidelinks]{hyperref}
\usepackage{amsmath}
\usepackage{amssymb}
\usepackage{amsthm}
\usepackage[sort&compress,numbers]{natbib}

\newtheorem{theorem}{Theorem}[section]
\newtheorem*{theorem*}{Theorem}

\newtheorem*{claim*}{Claim}

\newtheorem{proposition}[theorem]{Proposition}

\theoremstyle{definition}

\newtheorem{remark}[theorem]{Remark}
\newtheorem{lemma}[theorem]{Lemma}
\newtheorem{definition}[theorem]{Definition}
\AtBeginDocument{   \def\MR#1{}
}

\begin{document}
\title[A von Neumann algebra characterization of property (T) for groupoids]{%
A von Neumann algebra characterization\\
of property (T) for groupoids}
\author{Martino Lupini}
\address{Martino Lupini\\
Mathematics Department\\
California Institute of Technology\\
1200 E. California Blvd\\
MC 253-37\\
Pasadena, CA 91125}
\email{lupini@caltech.edu}
\urladdr{http://www.lupini.org/}
\date{\today }
\thanks{The author was partially supported by the NSF Grant DMS-1600186.}
\dedicatory{}

\begin{abstract}
For an arbitrary discrete probability-measure-preserving groupoid $G$, we provide a characterization of
property (T) for $G$ in terms of the groupoid von Neumann algebra $L(G)$.
More generally, we obtain a characterization of relative property (T) for a
subgroupoid $H\subset G$ in terms of the inclusion $L\left( H\right) \subset
L\left( G\right) $.
\end{abstract}

\maketitle


\section{Introduction}

Property (T) for countable probability-measure-preserving (pmp) equivalence
relations has been introduced by Zimmer in \cite{zimmer_cohomology_1981}.
The natural generalization to discrete pmp groupoids has been studied by
Anantharaman-Delaroche in \cite{ad_cohomology_2005}. In view of the
importance of property\ (T) in the setting of operator algebras, and the key
role it plays in Popa's deformation/rigidity theory, it is valuable to have
a characterization of property (T) for a discrete pmp groupoid $G$ solely in
terms of the inclusion $L^{\infty }(G^{0})\subset L(G)$, where $L(G)$ is the
groupoid von Neumann algebra of $G$ and $G^{0}$ is the unit space of $G$.
Such a characterization has been established by Connes and Jones in the case
when $G$ is a countable discrete group with infinite conjugacy classes, in
which case $L(G)$ is a II$_{1}$ factor. They showed in \cite%
{connes_property_1985} that, under those assumptions, $G$ has property (T)
if and only if $L(G)$ has property (T) in the sense defined therein. In the
case when $G$ is an ergodic II$_{1}$ equivalence relation, in which case $%
L(G)$ is a II$_{1}$ factor and $L^{\infty }(G^{0})$ is a Cartan subalgebra
of $L(G)$, such a characterization has been established by Popa in \cite%
{popa_correspondences_1986}. It is shown there that, under those
assumptions, $G$ has property (T) as defined by Zimmer if and only if the
inclusion $L(G^{0})\subset L(G)$ is co-rigid in the sense of \cite[Remark
2.6.1]{popa_class_2006}. The purpose of this paper is to provide a common
generalization of such characterizations, applicable to an arbitrary
discrete pmp groupoid $G$.

More generally, given a subgroupoid $H$ of $G$, we consider the natural
notion of \emph{relative }property (T) of $H$ in $G$, generalizing the usual
notion for groups. We then obtain a characterization of relative property
(T) of $H$ in $G$ in terms of the inclusion $L^{\infty }\left( X\right)
\subset L\left( H\right) \subset L(G)$, where $X$ is the common unit space
of $G$ and $H$. Again, such a characterization is applicable to an arbitrary
pair of discrete pmp groupoids. In the case of groups, relative property (T)
has been characterized in terms of group von Neumann algebra by Popa in \cite%
[Proposition 5.1]{popa_class_2006}. It is shown there that the pair $H\leq G$
has property (T) if and only if the inclusion $L\left( H\right) \subset
L\left( G\right) $ is rigid in the sense defined therein; see \cite[%
Definition 4.2]{popa_class_2006}.

More generally, we consider a notion of property (T) for a triple $K\leq
H\leq G$ consisting of a discrete pmp groupoid $G$ together with nested
subgroupoid $K,H$. Such a notion has been considered in the setting of
groups in \cite[Definition 2.3]{gardella_complexity_2017}. It subsumes property (T) for a
pair $H\leq G$ in the case when $K$ is the trivial subgroupoid of $G$, i.e.\
the unit space. Again we obtain, for an arbitrary such triple, a
characterization of property (T) in terms of the inclusions $L^{\infty
}\left( X\right) \subset L\left( K\right) \subset L\left( H\right) \subset
L(G)$. We also provide a cohomological characterization of such a notion,
and in particular of the notion of relative property (T). In the case of
property (T) for a single groupoid, such a characterization has been
obtained by Anantharaman-Delaroche in \cite{ad_cohomology_2005}.

The rest of this paper is divided into two sections, apart from this
introduction. In Section \ref{Section:T} we recall the fundamental notions
and definitions concerning groupoids to be used in the rest of the paper,
introduce the notion of property (T) for triples of groupoids, and obtain
the cohomological characterization mentioned above. In\ Section \ref%
{Section:rigid} we obtain a von Neumann algebra characterization of property
(T) for triples of groupoids, in terms of the groupoid von Neumann algebras.

Throughout the paper, we follow the convention that scalar products in
Hilbert spaces are linear in the second variable and conjugate-linear in the
first variable.

\section{Property (T) for groupoids\label{Section:T}}

\subsection{Groupoids}

A \emph{groupoid }is, briefly, a small category $G$ where every morphism is
invertible. In this case, the objects\emph{\ }of $G$ are also called \emph{%
units}, and the set of units is denoted by $G^{0}$. The morphisms in $G$ are
also called \emph{arrows}. As it is customary, we canonically identify every
object with the corresponding identity arrow. This allows one to regard $%
G^{0}$ as a subset of $G$. There are canonical source and range maps $%
s,r:G\rightarrow G^{0}$ that map each arrow $\gamma $ in $G$ to the objects $%
s(\gamma ),r(\gamma )\in G^{0}$ such that $\gamma $ is an arrow from $%
s(\gamma )$ to $r(\gamma )$. A pair of arrows $\left( \gamma ,\rho \right) $
is \emph{composable} if $s(\gamma )=r(\rho )$. The set of pairs of
composable arrows of $G$ is denoted by $G^{2}$. One can then regard
composition of arrows as a function $G^{2}\rightarrow G$, $\left( \gamma
,\rho \right) \mapsto \gamma \rho $. Since by assumption every arrow of $G$
is invertible, one can also consider the function $G\rightarrow G$, $\gamma
\mapsto \gamma ^{-1}$ that maps each arrow to its inverse. In the following,
given subsets $A,B$ of $G$, we let $AB$ be the set of arrows $\gamma \rho $
for $\left( \gamma ,\rho \right) \in G^{2}\cap \left( A\times B\right) $. If 
$\gamma \in G$, we also write $A\gamma $ and $\gamma A$ for $A\left\{ \gamma
\right\} $ and $\left\{ \gamma \right\} A$, respectively. Consistently, if $%
x $ is a unit of $G$, then $xA$ is the set of arrows in $A$ with range $x$,
and $Ax$ is the set of arrows in $A$ with source $x$.

A Borel groupoid is a groupoid $G$ endowed with a standard Borel structure
such that the set of objects is Borel, and composition and inversion of
arrows are Borel maps. A countable Borel groupoid is a standard Borel
groupoid such that $xG$ and $Gx$ are countable sets for every $x\in G^{0}$
or, equivalently, source and range maps are countable-to-one. In the
following we will tacitly use classical Borel selection theorem for
countable-to-one Borel maps as can be found in \cite[Section 18.C]%
{kechris_classical_1995}. A discrete probability-measure-preserving (pmp)
groupoid is a pair $\left( G,\mu \right) $ where $G$ is a countable Borel
groupoid and $\mu $ is a Borel probability measure on $G^{0}$ satisfying%
\begin{equation*}
\int_{x\in G^{0}}\left\vert xA\right\vert d\mu (x)=\int_{x\in
G^{0}}\left\vert Ax\right\vert d\mu (x)
\end{equation*}%
for every Borel subset $A$ of $G^{0}$. In such a case this expression
defines an extension of $\mu $ to a $\sigma $-finite Borel measure defined
on the whole of $G$. In the following we regard a discrete pmp groupoid $G$
as a measure space endowed with such a measure. One can also define a
canonical measure on $G^{2}$ given by%
\begin{equation*}
\mu _{G^{2}}(A)=\int_{x\in G^{0}}\left( A\cap \left( Gx\times xG\right)
\right) d\mu (x)\text{.}
\end{equation*}%
Given a non-null Borel subset $A$ of $G^{0}$ one can define the \emph{%
reduction }$G|_{A}$ to be the groupoid $AGA$ with set of objects $A$ endowed
with the measure $\mu _{A}:=\frac{1}{\mu (A)}\mu $. Such a reduction is
called \emph{inessential }if $A$ is conull. In the following, we identify
two discrete pmp groupoids whenever they have isomorphic inessential
reductions. A Borel subset $A$ of $G$ is invariant if $r\left( GA\right) =A$%
. The groupoid $G$ is ergodic if every invariant set $A\subset G^{0}$ is
either null or conull.

Suppose that $G$ is a discrete pmp groupoid, and $H$ is a standard Borel
groupoid. A \emph{homomorphism }from $G$ to $H$ is a Borel map $%
f:G\rightarrow H$ satisfying $s\left( f(\gamma )\right) =f\left( s(\gamma
)\right) $ and $r\left( f(\gamma )\right) =f\left( r(\gamma )\right) $ for
a.e.\ $\gamma \in G$, and $f\left( \gamma \rho \right) =f(\gamma )f(\rho )$
for a.e.\ $\left( \gamma ,\rho \right) \in G^{2}$. This is equivalent to the
assertion that there exists a conull Borel subset $A$ of $G^{0}$ such that $%
s\left( f(\gamma )\right) =f\left( s(\gamma )\right) $, $r\left( f(\gamma
)\right) =f\left( r(\gamma )\right) $, and $f\left( \gamma \rho \right)
=f(\gamma )f(\rho )$ for ever $\gamma ,\rho \in AGA$ \cite[Lemma 5.2]%
{ramsay_virtual_1971}. A subgroupoid of $G$ is a Borel subset $H$ of $G$
which is also a groupoid, such that $G$ and $H$ have the same unit space $%
\left( X,\mu \right) $, and the inclusion map $H\subset G$ is a homomorphism
from $H$ to $G$.

A (Borel) \emph{bisection }of a discrete pmp groupoid $G$ is a Borel subset $%
t$ of $G$ such that $xt$ and $tx$ have size at most $1$ for every $x\in G^{0}
$. Borel bisections naturally form an inverse semigroup with respect to the
operation 
\begin{equation*}
\left( t_{0},t_{1}\right) \mapsto t_{0}t_{1}=\left\{ \gamma \rho :\left(
\gamma ,\rho \right) \in \left( t_{0}\times t_{1}\right) \cap G^{2}\right\} 
\text{.}
\end{equation*}%
The full pseudogroup $\left[ \left[ G\right] \right] $ of $G$ is the inverse
semigroup consisting of Borel bisections modulo the relation of being equal
almost everywhere. In the following we will always identify Borel bisections
when they agree almost everywhere. The full group $\left[ G\right] $ is the
subset of $\left[ \left[ G\right] \right] $ consisting of the Borel
bisections $t$ such that $tt^{-1}=t^{-1}t=G^{0}$. This is a Polish group
with respect to the topology induced by the metric $d\left(
t_{0},t_{1}\right) =\mu \left( t_{0}\bigtriangleup t_{1}\right) $. A
countable subgroup $\Gamma $ of $\left[ G\right] $ \emph{covers }$G$ if $%
G=\bigcup \Gamma $. If $G$ is ergodic and $A,B$ are Borel subsets of $G^{0}$%
, then $\mu (A)=\mu (B)$ if and only if there exists $t\in \left[ G\right] $
such that $B=r\left( tA\right) $.

Clearly, any countable discrete group is, in particular, a discrete pmp
groupoids. Indeed, these are precisely the discrete pmp groupoids whose unit
space contains a single element. At the opposite end of the spectrum, every
countable pmp equivalence relation is a discrete pmp groupoid. Indeed, these
are precisely the discrete pmp groupoids $G$ which are \emph{principal}, in
the sense that the function $G\rightarrow G^{0}\times G^{0}$, $\gamma
\mapsto \left( r(\gamma ),s(\gamma )\right) $ is one-to-one. Thus, the class
of discrete pmp groupoids subsumes both countable discrete groups and
countable pmp equivalence relations.

\subsection{Representations of groupoids}

Suppose that $X$ is a standard probability space. A standard Borel space
fibered over $X$ is a standard Borel space $Z$ endowed with a distinguished
Borel map $p:Z\rightarrow X$. In this case we let, for $x\in X$, $%
Z_{x}:=p^{-1}\left\{ x\right\} $ be the corresponding \emph{fiber }over $x$.
We denote the space $Z$ also by $\bigsqcup_{x\in X}Z_{x}$. Given two
standard Borel spaces $Z,Z^{\prime }$ fibered over $X$ one can define the 
\emph{fibered product} 
\begin{equation*}
Z\ast Z^{\prime }=\left\{ \left( z,z^{\prime }\right) \in Z_{x}\times
Z_{x}^{\prime }:x\in X\right\} \subset Z\times Z^{\prime }\text{,}
\end{equation*}%
which is still a standard Borel space fibered over $X$. A fibered map $f$
from $Z$ to $Z^{\prime }$ is a Borel function that maps $Z_{x}$ to $%
Z_{x}^{\prime }$ for $x\in X$. If $Y$ is a standard Borel space, when we
regard $Y\times X$ as a standard Borel space fibered over $X$ with respect
to the product Borel structure and the projection to the second factor. In
particular, we regard $X$ as a standard Borel space fibered over itself via
the identity map. A section $\sigma $ for a standard Borel space $Z$ fibered
over $X$ is a fibered map from $X$ to $Z$. In this case, we let $\sigma _{x}$
be the value of $\sigma $ at $x\in X$.

A (Borel, complex) Hilbert bundle over $X$ is a standard Borel space $%
\mathcal{H}$ fibered over $X$ endowed fibered functions $0:X\rightarrow 
\mathcal{H}$ (zero section), $+:\mathcal{H}\ast \mathcal{H}\rightarrow 
\mathcal{H}$ (sum), and $\mathbb{C}\times \mathcal{H}\rightarrow \mathcal{H}$
(scalar multiplication) that define on each fiber $\mathcal{H}_{x}$ for $%
x\in X$ a (complex) vector space structure, and such that there exists a
sequence of sections $\left( \sigma _{n}\right) _{n\in \mathbb{N}}$ of $%
\mathcal{H}$ such that $\left\{ \sigma _{n,x}:n\in \mathbb{N}\right\} $ has
dense linear span in $\mathcal{H}_{x}$. The Gram-Schmidt orthonormalization
process shows that one can furthermore assume that $\left\{ \sigma
_{n,x}:n\in \mathbb{N}\right\} $ is an orthonormal basis of $\mathcal{H}_{x}$
for $x\in X$. In this case, we call $\left( \sigma _{n}\right) _{n\in 
\mathbb{N}}$ an \emph{orthonormal basic sequence }for $\mathcal{H}$. The 
\emph{unitary groupoid }$U(\mathcal{H})$ is the groupoid consisting of the
unitary operators $U:\mathcal{H}_{x}\rightarrow \mathcal{H}_{y}$ for $x,y\in
X$. This is a standard Borel groupoid when endowed with the standard Borel
structure generated by the source and range maps together with the functions 
$\left( U:\mathcal{H}_{x}\rightarrow \mathcal{H}_{y}\right) \mapsto
\left\langle \sigma _{n,y},U\sigma _{m,x}\right\rangle $ for $n,m\in \mathbb{%
N}$. The unit space of $U(\mathcal{H})$ can be identified with $X$. One can
also consider the space $L^{2}\left( X,\mathcal{H}\right) $ of sections for $%
\mathcal{H}$ satisfying%
\begin{equation*}
\int \left\Vert \xi _{x}\right\Vert ^{2}d\mu (x)<+\infty
\end{equation*}%
identified when they agree almost everywhere. This is a Hilbert space with
respect to the inner product 
\begin{equation*}
\left\langle \xi ,\eta \right\rangle =\int \left\langle \xi _{x},\eta
_{x}\right\rangle d\mu (x)
\end{equation*}%
for $\xi ,\eta \in L^{2}\left( X,\mathcal{H}\right) $.

Suppose that $G$ is a discrete pmp groupoid, and $\mathcal{H}$ is a
(complex) Hilbert bundle over $G^{0}$. A (unitary) representation of $\pi $
on $\mathcal{H}$ is a homomorphism from $G$ to $U(\mathcal{H})$ that is the
identity on the unit space. An invariant sub-bundle of $\mathcal{H}$ is a
Borel subset $\mathcal{K}$ of $\mathcal{H}$ such that $\mathcal{K}_{x}$ is a
subspace of $\mathcal{H}$ for a.e.\ $x\in G^{0}$, and $\pi _{\gamma }$ maps $%
\mathcal{K}_{s(\gamma )}$ onto $\mathcal{K}_{r(\gamma )}$ for a.e.\ $\gamma
\in G$. When $G$ is ergodic, the Borel function $x\mapsto \dim \mathcal{K}%
_{x}$ is constant almost everywhere, one can define its constant value to be
the \emph{dimension} of $\mathcal{K}$. Real Hilbert bundles and (orthogonal)
representations of discrete pmp groupoids on real Hilbert bundles can be
defined in a similar fashion.

A representation $\pi $ of $G$ on $\mathcal{H}$ induces a representation $%
\left[ \left[ \pi \right] \right] $ of the inverse semigroup $\left[ \left[ G%
\right] \right] $ on $L^{2}\left( G^{0},\mathcal{H}\right) $. This is
defined by setting 
\begin{equation*}
\left( \left[ \left[ \pi \right] \right] _{\sigma }\xi \right) _{x}=\left\{ 
\begin{array}{ll}
\pi _{x\sigma }\xi _{s\left( x\sigma \right) } & \text{if }x\in \sigma
\sigma ^{-1} \\ 
0 & \text{otherwise.}%
\end{array}%
\right.
\end{equation*}%
In particular the restriction $\left[ \pi \right] $ of $\left[ \left[ \pi %
\right] \right] $ to $\left[ G\right] $ is a continuous representation of
the Polish group $\left[ G\right] $.

A \emph{unit section }for $\mathcal{H}$ is a section $\xi $ such that $%
\left\Vert \xi _{x}\right\Vert =1$ for a.e.\ $x\in G^{0}$. A unit section $%
\xi $ for $\mathcal{H}$ is \emph{invariant} if $\pi _{\gamma }\xi _{s(\gamma
)}=\xi _{r(\gamma )}$ for a.e.\ $\gamma \in G$. We say that the
representation $\pi $ of an ergodic discrete pmp groupoid $G$ on $\mathcal{H}
$ is \emph{ergodic }if it has no invariant unit sections. This is equivalent
to the assertion that for some (equivalently, every) countable subgroup $%
\Gamma $ of $\left[ G\right] $ that covers $G$, $\left[ \pi \right]
|_{\Gamma }$ is ergodic. Let $\xi $ be a unit section for $\mathcal{H}$. We
say that $\xi $ is \emph{cyclic} if, for a.e.\ $x\in G^{0}$ one has that $%
\left\{ \pi _{\gamma }\xi _{s(\gamma )}:\gamma \in xG\right\} $ has dense
linear span in $\mathcal{H}_{x}$.

Let $\overline{\mathcal{H}}:=\bigsqcup_{x\in \mathcal{H}_{x}}\overline{%
\mathcal{H}}_{x}$, where $\overline{\mathcal{H}}_{x}$ denotes the conjugate
Hilbert space of $\mathcal{H}_{x}$, with canonical conjugate linear
isomorphism $\mathcal{H}_{x}\rightarrow \overline{\mathcal{H}}_{x}$, $\xi
\mapsto \overline{\xi }_{x}$. The \emph{conjugate representation }$\overline{%
\pi }$ of $G$ on $\overline{\mathcal{H}}$ is defined by $\overline{\pi }%
_{\gamma }\overline{\xi }=\overline{\pi _{\gamma }\xi }$ for $\gamma \in G$
and $\xi \in \mathcal{H}_{s(\gamma )}$.

Suppose that $\pi _{0}$ and $\pi _{1}$ are representations of $G$ on Hilbert
bundles $\mathcal{H}_{0}$ and $\mathcal{H}_{1}$.\ Then one can consider the
Hilbert bundle $\mathcal{H}_{0}\otimes \mathcal{H}_{1}:=\bigsqcup_{x\in
G^{0}}\mathcal{H}_{0,x}\otimes \mathcal{H}_{1,x}$ and the representation $%
\pi _{0}\otimes \pi _{1}$ of $G$ on $\mathcal{H}_{0}\otimes \mathcal{H}_{1}$
defined in the obvious way.

\begin{remark}
\label{Remark:extend}In the following we will frequently use the following
observation. Suppose that $G$ is an ergodic discrete pmp groupoid, $A\subset
G^{0}$ is a non-null Borel set, and $\pi $ is a representation of $G$ on $%
\mathcal{H}$. Then one can consider the representation $\pi _{A}$ of $G|_{A}$
on $\mathcal{H}|_{A}=\bigsqcup_{x\in A}\mathcal{H}_{x}$ obtained from $\pi $
by restriction. If $\eta $ is an invariant section for $\mathcal{H}|_{A}$,
then there exists a unique invariant section $\xi $ for $G$ such that $\xi
_{x}=\eta _{x}$ for $x\in A$. In particular, $\pi $ is ergodic if and only
if $\pi _{A}$ is ergodic.
\end{remark}

The notion of \emph{weak mixing} representation has been introduced in \cite[Definition 3.11]{gardella_complexity_2017}. The representation $\pi $ of an ergodic groupoid 
$G$ on $\mathcal{H}$ is weak mixing if, for every $\varepsilon >0$, $n\in 
\mathbb{N}$, and sections $\xi _{1},\ldots ,\xi _{n}$ for $\mathcal{H}$
there exists $t\in \left[ G\right] $ such that, for every $i,j\in \left\{
1,2,\ldots ,n\right\} $,%
\begin{equation*}
\int_{G^{0}}\left\vert \left\langle \xi _{j,x},\pi _{xt}\xi _{i,s\left(
xt\right) }\right\rangle \right\vert d\mu (x)\leq \varepsilon \text{.}
\end{equation*}%
Several equivalent characterizations of such a notion have been established
in \cite[Subsection 3.3]{gardella_complexity_2017}. Particularly, a representation $\pi $ of
an ergodic discrete pmp groupoid $G$ on $\mathcal{H}$ is weak mixing if and
only if for some (equivalently, every) countable subgroup $\Gamma $ of $%
\left[ G\right] $ that covers $G$, $\left[ \pi \right] |_{\Gamma }$ is weak
mixing, if and only if $\mathcal{H}$ does not have an finite-dimensional
invariant sub-bundle, if and only if $\pi \otimes \overline{\pi }$ is
ergodic.

\subsection{Property (T)}

Suppose that $\Gamma $ is a countable discrete group, and $\pi $ is a
representation of $\Gamma $ on a Hilbert space $\mathcal{H}$. If $F$ is a
subset of $\Gamma $, and $\varepsilon >0$, then a unit vector $\xi $ of $%
\mathcal{H}$ is $\left( F,\varepsilon \right) $-invariant if it satisfies $%
\left\Vert \pi _{\gamma }\xi -\xi \right\Vert \leq \varepsilon $ for every $%
\gamma \in F$. The representation $\pi $ has almost invariant vectors if,
for every finite subset $F$ of $\Gamma $ and for every $\varepsilon >0$, it
has an $\left( F,\varepsilon \right) $-invariant vector. The group $\Gamma $
has property (T) if every representation of $\Gamma $ that has almost
invariant unit vectors, it has an invariant unit vector \cite%
{bekka_kazhdans_2008}. A standard reference for the theory of property (T)
groups is \cite{bekka_kazhdans_2008}.

The notion of property (T) for pmp equivalence relations has been introduced
in \cite{zimmer_cohomology_1981}. A natural common generalization of the
notion of property (T) for discrete groups and pmp equivalence relations has
been considered in \cite{ad_cohomology_2005}. Let $G$ be a discrete pmp
groupoid with unit space $X$, and $\pi $ be a representation of $G$ on a
Hilbert bundle $\mathcal{H}$. If $F$ is a subset of $\left[ G\right] $ and $%
\varepsilon >0$, then we say that a unit section $\xi $ for $\mathcal{H}$ is 
$\left( F,\varepsilon \right) $-invariant if it satisfies $\left\Vert \left[
\pi \right] _{t}\xi -\xi \right\Vert \leq \varepsilon $ in $L^{2}\left( X,%
\mathcal{H}\right) $ for every $t\in F$. The representation $\pi $ has
almost invariant unit sections if, for every finite subset $F$ of $\left[ G%
\right] $ and $\varepsilon >0$, it has an $\left( F,\varepsilon \right) $%
-invariant unit section. The discrete pmp groupoid $G$ has property\ (T) if
every representation $\pi $ of $G$ that has almost invariant unit sections
it has an invariant unit section \cite[Definition 4.3]{ad_cohomology_2005}.

The notion of property (T) admits a natural \emph{relative }version for
subgroups. Suppose that $\Gamma $ is a countable discrete group, and $%
\Lambda \leq \Gamma $ is a subgroup. If $\pi $ is a representation of $%
\Gamma $ on a Hilbert space $\mathcal{H}$, then a unit vector $\xi $ in $%
\mathcal{H}$ is $\Lambda $-invariant if it is invariant for the restriction
of $\pi $ to $\Lambda $. Then $\Lambda $ has relative property (T) in $%
\Gamma $, or the pair $\Lambda \leq \Gamma $ has property (T), if every
representation of $\Gamma $ that has almost invariant unit vectors, it has a 
$\Lambda $-invariant unit vector. This notion admits a natural
generalization to discrete pmp groupoids. Suppose that $G$ is a discrete pmp
groupoid, $H$ is a subgroupoid of $G$, and $\pi $ is a unitary
representation of $G$ on a Hilbert bundle $\mathcal{H}$. Then a unit section 
$\xi $ for $\mathcal{H}$ is $H$-invariant if it is invariant for the
restriction of $\pi $ to $H$. Then $H$ has the relative property (T) in $G$,
or pair $H\leq G$ has property\ (T), if every representation of $G$ that has
almost invariant unit sections admits an $H$-invariant unit section.
Clearly, when $H=G$, this recovers property (T) for a single discrete pmp
groupoid.

A natural generalization of property (T) from pairs of groups to triples of
groups has been considered in \cite[Definition 2.3]{gardella_complexity_2017}. Suppose that $%
\Gamma $ is a countable discrete group, and $\Delta \leq \Lambda \leq \Gamma 
$ are nested subgroups. Then the triple $\Delta \leq \Lambda \leq \Gamma $
has property\ (T) if every representation of $\Gamma $ with almost invariant 
$\Delta $-invariant unit vectors admits a $\Lambda $-invariant unit vector.
Clearly, when $\Delta $ is the trivial subgroup one recovers the notion of
property (T) for pairs of groups. Naturally, one can generalize such a
notion to discrete pmp groupoids as follows. Suppose that $G$ is a discrete
pmp groupoid, and $K\leq H\leq G$ are nested subgroupoids.

\begin{definition}
\label{Definition:invariant-T-groupoid}The triple $K\leq H\leq G$ has
property (T) if, for every representation $\pi $ of $G$, if $\pi $ has
almost invariant $K$-invariant unit sections, then $\pi $ has an $H$%
-invariant unit section.
\end{definition}

\begin{remark}
When $H$ is ergodic, in Definition \ref{Definition:invariant-T-groupoid} one
can equivalently require that every representation of $G$ with almost
invariant $K$-invariant unit sections has a nonzero $H$-invariant unit
section $\xi $. Indeed, in this case one has that there exists $\delta >0$
such that $\left\Vert \xi _{x}\right\Vert =\delta $ for a.e.\ $x\in H^{0}$.
Therefore $\delta ^{-1}\xi $ is a $H$-invariant unit section.
\end{remark}

Again, when $K$ is the trivial subgroupoid of $H$---i.e.\ $K$ is equal to
the common unit space of $H$ and $G$---one recovers the notion of property
(T) for pairs $H\leq G$.

Several equivalent characterizations of property (T) for pairs of groups are
established in \cite{jolissaint_property_2005}. Furthermore, a cohomological
characterization of property (T) for discrete pmp groupoids is the main
result of \cite{ad_cohomology_2005}. In this section we provide a
characterization for triples of discrete pmp groupoids, subsuming the
characterizations from \cite{jolissaint_property_2005,ad_cohomology_2005}.

\subsection{Cohomology of representations\label{Subsection:cohomology}}

Let $G$ be a discrete pmp groupoid, with subgroupoids $K\leq H\leq G$.
Denote by $X$ their common unit space. In the space of complex-valued Borel
functions on $G$ consider the (Polish) topology generated by the
pseudometrics%
\begin{equation*}
d_{t}\left( \varphi ,\varphi ^{\prime }\right) :=\int_{x\in X}\frac{%
\left\vert \varphi \left( tx\right) -\varphi ^{\prime }\left( tx\right)
\right\vert }{1+\left\vert \varphi \left( tx\right) -\varphi ^{\prime
}\left( tx\right) \right\vert }d\mu _{G^{0}}(x)
\end{equation*}%
where $t$ ranges within (a dense subset of) $\left[ G\right] $. If $\mathcal{%
H}$ is a Hilbert bundle over $X$, then we let $S\left( G,\mathcal{H}\right) $
be the space of Borel functions $G\rightarrow \mathcal{H}$, $\gamma \mapsto
b_{\gamma }\in \mathcal{H}_{r(\gamma )}$ endowed with the topology generated
by the pseudometrics%
\begin{equation*}
d_{t}\left( b,b^{\prime }\right) =\int_{x\in X}\frac{\left\Vert
b_{tx}-b_{tx}^{\prime }\right\Vert }{1+\left\Vert b_{tx}-b_{tx}^{\prime
}\right\Vert }d\mu _{G^{0}}(x)
\end{equation*}

Suppose that $\pi $ is a representation of $G$ on the Hilbert bundle $%
\mathcal{H}$. A \emph{cocycle }for $\pi $ is an element $b$ of $S\left( G,%
\mathcal{H}\right) $ such that $b_{\gamma _{1}\gamma _{2}}=b_{\gamma
_{1}}+\pi _{\gamma _{1}}\left( b_{\gamma _{2}}\right) $ for a.e.\ $\left(
\gamma _{1},\gamma _{2}\right) \in G^{2}$. A cocycle $b$ for $\pi $ is $K$-%
\emph{trivial} if $b_{\gamma }=0$ for a.e.\ $\gamma \in K$. A section $\xi $
for $\mathcal{H}$ defines a cocycle $c_{\pi }\left( \xi \right) $ for $\pi $
by setting $c_{\pi }\left( \xi \right) _{\gamma }=\xi _{r(\gamma )}-\pi
_{\gamma }\xi _{s(\gamma )}$ for $\gamma \in G$. Cocycles of this form are
called \emph{coboundaries}. The section $\xi $ is $K$-invariant if and only
if $c_{\pi }\left( \xi \right) $ is $K$-trivial. We denote the space of $K$%
-trivial cocycles for $\pi $ by $Z_{:K}^{1}\left( \pi \right) $, and the
space of $K$-trivial coboundaries for $\pi $ by $B_{:K}^{1}\left( \pi
\right) $.\ We let $Z_{:K,H}^{1}\left( \pi \right) $ be the set of \emph{%
restrictions} to $H$ of elements of $Z_{:K}^{1}\left( \pi \right) $. The $K$%
-invariant $H$-relative cohomology group $H_{:K,H}^{1}\left( \pi \right) $
of $\pi $ is the quotient of $Z_{:K,H}^{1}\left( \pi \right) $ by the
subgroup $B_{:K}^{1}\left( \pi |_{H}\right) $. The same argument as in \cite[%
Proposition 3.9]{ad_cohomology_2005} shows that $Z_{:K,H}^{1}\left( \pi
\right) $ is a closed subset of $S\left( H,\mathcal{H}\right) $.

The following result is established in \cite[Theorem 3.19]%
{ad_cohomology_2005}.

\begin{theorem}[Anantharaman-Delaroche]
\label{Theorem:ad-coboundaries}Suppose that $G$ is an ergodic discrete pmp
groupoid. Consider a representation $\pi $ of $G$ and a cocycle $b$ for $\pi 
$. The following assertions are equivalent:

\begin{enumerate}
\item $\pi $ is a coboundary;

\item there exists a non-null Borel subset $A$ of $G^{0}$ such that the
function $AGA\rightarrow \mathbb{R}$, $\gamma \mapsto \left\Vert b_{\gamma
}\right\Vert $ is bounded;

\item there exists a non-null Borel subset $A$ of $G^{0}$ such that, for
every $x\in A$, the function $AGx\rightarrow \mathbb{R}$, $\gamma \mapsto
\left\Vert b_{\gamma }\right\Vert $ is bounded.
\end{enumerate}
\end{theorem}

\subsection{Functions of positive and negative type}

Suppose that $G$ is a discrete pmp groupoid and $K\subset G$ is a
subgroupoid. The following is a standard definition; see \cite[Definition
4.1.2]{renault_c*-algebras_2009}.

\begin{definition}
A complex-valued function $\varphi :G\rightarrow \mathbb{C}$ is of \emph{%
positive type} if it is Borel, and for a.e.\ $x\in X$, for every $n\geq 1$, $%
\gamma _{1},\ldots ,\gamma _{n}\in xG$, and $\lambda _{1},\ldots ,\lambda
_{n}\in \mathbb{C}$ one has that%
\begin{equation*}
\sum_{ij}\overline{\lambda }_{i}\lambda _{j}\varphi \left( \gamma
_{i}^{-1}\gamma _{j}\right) \geq 0\text{.}
\end{equation*}%
A real-valued function $\varphi :G\rightarrow \mathbb{R}$ is of positive
type if it is Borel, $\varphi (\gamma )=\varphi (\gamma ^{-1})$ for a.e.\ $%
\gamma \in G$, and for a.e.\ $x\in X$, for every $n\geq 1$, $\gamma
_{1},\ldots ,\gamma _{n}\in xG$, and $\lambda _{1},\ldots ,\lambda _{n}\in 
\mathbb{R}$ one has that%
\begin{equation*}
\sum_{ij}\overline{\lambda }_{i}\lambda _{j}\varphi \left( \gamma
_{i}^{-1}\gamma _{j}\right) \geq 0\text{.}
\end{equation*}

The function $\varphi $ is $K$-\emph{invariant} if $\varphi \left( \rho
_{0}\gamma \right) =\varphi (\gamma )=\varphi \left( \gamma \rho _{1}\right) 
$ for every $\rho _{0},\rho _{1}\in K$ and $\gamma \in G$ such that $\left(
\rho _{0},\gamma \right) ,\left( \gamma ,\rho _{1}\right) \in G^{2}$. The
function $\varphi $ is called \emph{normalized }if $\varphi (x)=1$ for a.e.\ 
$x\in G^{0}$.
\end{definition}

The same proof as \cite[Proposition 5.3]{harpe_propriete_1989} gives the
following.

\begin{proposition}
Suppose that $\varphi $ is a Borel complex-valued (respectively,
real-valued) function on $G$. The following assertions are equivalent:

\begin{enumerate}
\item $\varphi $ is a normalized $K$-invariant function of positive type;

\item there exists a representation $\pi ^{\varphi }$ of $G$ on a complex
(respectively, real) Hilbert bundle $\mathcal{H}^{\varphi }$ and a $K$%
-invariant cyclic unit section $\xi ^{\varphi }$ for $\mathcal{H}^{\varphi }$
such that $\varphi (\gamma )=\left\langle \xi _{r(\gamma )},\pi _{\gamma
}\xi _{s(\gamma )}\right\rangle $ for a.e.\ $\gamma \in G$.
\end{enumerate}

The representation $\left( \pi ^{\varphi },\mathcal{H}^{\varphi },\xi
^{\varphi }\right) $ of $G$ is uniquely determined up to isomorphism, and it
will called the GNS representation of $\varphi $.
\end{proposition}

The following definition is considered in \cite[Proposition 5.19]%
{ad_cohomology_2005}.

\begin{definition}
A real-valued function $\psi :G\rightarrow \mathbb{R}$ on $G$ is of \emph{%
conditionally negative type} if it is Borel, $\psi (\gamma ^{-1})=\psi
(\gamma )$ for a.e.\ $\gamma \in G$, $\psi (x)=0$ for a.e.\ $x\in G^{0}$, and%
\begin{equation*}
\sum_{i,j=1}^{n}\overline{\lambda }_{i}\lambda _{j}\psi \left( \gamma
_{i}^{-1}\gamma _{j}\right) \leq 0
\end{equation*}%
for a.e.\ $x\in G^{0}$, for every $n\geq 2$, $\lambda _{1},\ldots ,\lambda
_{n}\in \mathbb{R}$ satisfying $\lambda _{1}+\cdots +\lambda _{n}=0$, and
every $\gamma _{1},\ldots ,\gamma _{n}\in xG$.

A complex-valued function $\psi :G\rightarrow \mathbb{C}$ is of
conditionally negative type if it satisfies the same properties where one
consider complex scalars instead of real scalars.
\end{definition}

The following proposition is essentially established in \cite[Proposition
5.21]{ad_cohomology_2005}.

\begin{proposition}
\label{Proposition:nt}Suppose that $\psi $ is a Borel real-valued function
on $G$. The following assertions are equivalent:

\begin{enumerate}
\item $\psi $ is a $K$-invariant function of conditionally negative type;

\item there exists a real Hilbert bundle $\mathcal{H}^{\psi }$, a
representation $\pi ^{\psi }$ of $G$ on $\mathcal{H}^{\psi }$, and a $K$%
-trivial cocycle $b^{\psi }$ for $\pi ^{\psi }$ such that $\{b_{\gamma
}^{\psi }:\gamma \in xG\}$ has dense linear span in $\mathcal{H}^{\psi }$,
and such that $\psi (\gamma )=\left\Vert b_{\gamma }^{\psi }\right\Vert ^{2}$
for a.e.\ $\gamma \in G$.
\end{enumerate}

Furthermore one has that $\psi \left( \rho ^{-1}\gamma \right) =\left\Vert
b_{\gamma }^{\psi }-b_{\rho }^{\psi }\right\Vert ^{2}$ for $\gamma ,\rho \in
G$ such that $r(\gamma )=r(\rho )$.
\end{proposition}

The following two lemmas are consequences of \cite[Theorem\ C.3.2]%
{bekka_kazhdans_2008} and \cite[Proposition 5.18]{harpe_propriete_1989}.

\begin{lemma}
\label{Lemma:exponential-negative-type}Suppose that $\psi :G\rightarrow 
\mathbb{R}$ is a Borel real-valued function such that $\psi (x)=0$ for a.e.\ 
$x\in X$ and $\psi (\gamma )=\psi (\gamma ^{-1})$ for a.e.\ $\gamma \in G$.\
Then the following assertions are equivalent:

\begin{enumerate}
\item $\psi $ is conditionally of negative type;

\item the function $\gamma \mapsto \exp \left( -t\psi (\gamma )\right) $ is
of positive type for every $t>0$.
\end{enumerate}
\end{lemma}

\begin{lemma}
\label{Lemma:real-part-negative-type}Suppose that $\psi :G\rightarrow 
\mathbb{C}$ is a complex-valued function of conditionally negative type.
Then $\mathrm{Re}\left( \psi \right) $ is a real-valued function of
conditionally positive type. Furthermore, $\psi $ is bounded if and only if $%
\mathrm{Re}\left( \psi \right) $ is bounded and, for every $x\in G^{0}$, $%
\mathrm{Re}\left( \psi \right) |_{Gx}$ is bounded if and only if $\psi
|_{Gx} $ is bounded.
\end{lemma}

\subsection{A cohomological characterization}

We now provide a characterization of property (T) for triples of groupoids,
including in particular a cohomological characterization; see Theorem \ref%
{Theorem:T} below. Such a cohomological characterization generalizes the one
in \cite{ad_cohomology_2005} for single groupoids. Even in this case, some
parts of the proof presented here are different, and in fact closer in
spirit to the group case as in \cite{jolissaint_property_2005}.

\begin{lemma}
\label{Lemma:bounded}Suppose that $G$ is a discrete pmp groupoid with unit
space $\left( X,\mu \right) $, and $H$ is an ergodic subgroupoid of $G$. Let 
$\psi :G\rightarrow \mathbb{R}$ be a function of conditionally negative
type. For $t>0$ let $\pi ^{(t)}$ be the representation on the Hilbert bundle 
$\mathcal{H}^{(t)}$ and $\xi ^{(t)}$ be the section of $\mathcal{H}^{(t)}$
obtained from the function of positive type $\exp \left( -t\psi \right) $
via the GNS construction. The following assertions are equivalent:

\begin{enumerate}
\item there exists a non-null Borel subset $A$ of $X$ such that, for a.e.\ $%
x\in A$, $\psi |_{AHx}$ is bounded;

\item for every $t>0$, $\pi ^{(t)}|_{H}$ is not ergodic;

\item there exists $t>0$ such that $\pi ^{(t)}|_{H}$ is not weak mixing;

\item for every non-null Borel subset $B$ of $X$ there exists a non-null
Borel subset $A$ of $X$ contained in $B$ such that, for a.e.\ $x\in A$, $%
\psi |_{AHx}$ is bounded;

\item for every non-null Borel subset $B$ of $X$, there exists a non-null
Borel subset $A$ of $X$ contained in $B$ such that $\psi |_{AHA}$ is bounded.
\end{enumerate}
\end{lemma}

\begin{proof}
(1)$\Rightarrow $(2) Suppose that $\psi |_{AHx}$ is bounded by $c_{x}$ for
a.e.\ $x\in A$, where $A$ is a non-null Borel subset of $X$. In view of
Remark \ref{Remark:extend}, after replacing $G$ with $G|_{A}$, we can assume
that $A=X$. Set $c(t):=\int \exp \left( -tc_{x}\right) d\mu (x)$ for $t>0$.
Define $C$ to be the closed convex hull of $\left\{ \left[ \pi ^{(t)}\right]
\xi ^{(t)}:\sigma \in \left[ H\right] \right\} $. We claim that $\left\Vert
\xi \right\Vert \geq c(t)$ for every $\xi \in C$. It is enough to consider
the case when $\xi =\sum_{i=1}^{n}s_{i}\left[ \pi ^{(t)}\right] _{\sigma
_{i}}\xi ^{(t)}$ for $\sigma _{i}\in \left[ H\right] $ and $s_{i}\in \left[
0,1\right] $ such that $s_{1}+\cdots +s_{n}=1$. In this case, we have that%
\begin{eqnarray*}
\left\Vert \sum_{i}s_{i}[\pi ^{(t)}]_{\sigma _{i}}\xi ^{(t)}\right\Vert ^{2}
&=&\sum_{ij}s_{i}s_{j}\left\langle \xi ^{(t)},[\pi ^{(t)}]_{\sigma
_{i}^{-1}\sigma _{j}}\xi ^{(t)}\right\rangle d\mu (x) \\
&=&\sum_{ij}s_{i}s_{j}\int \exp \left( -t\psi \left( \sigma _{i}^{-1}\sigma
_{j}x\right) \right) d\mu (x) \\
&\geq &\sum_{ij}s_{i}s_{j}\int \exp \left( -tc_{x}\right) d\mu (x) \\
&=&c(t)\text{.}
\end{eqnarray*}%
Pick now an element $\xi $ of $C$ of minimal norm. Observe that $\xi $ is
nonzero, and it is $H$-invariant by uniqueness.

(2)$\Rightarrow $(3) Obvious.

(3)$\Rightarrow $(1) Suppose that (1) does not hold. Thus for every non-null
Borel subset $A$ of $X$, there exists a non-null Borel subset $B$ such that
for every $x\in B$, $\psi |_{AHx}$ is unbounded. Fix $t>0$. We claim that
this implies that, for any unit section $\eta $ for $\mathcal{H}%
^{(t)}\otimes \overline{\mathcal{H}}^{(t)}$ and $\varepsilon \in \left(
0,1\right) $ there exists $\rho \in \left[ H\right] $ such that one has that 
\begin{equation*}
\left\vert \left\Vert \lbrack \pi ^{(t)}\otimes \overline{\pi }^{(t)}]_{\rho
}\eta -\eta \right\Vert ^{2}-2\right\vert <\varepsilon \text{.}
\end{equation*}%
In particular this shows that $\left( \pi ^{(t)}\otimes \overline{\pi }%
^{(t)}\right) |_{H}$ is ergodic, and hence $\pi ^{(t)}|_{H}$ is weak mixing.
Since $\xi ^{(t)}$ is a cyclic unit section for $\mathcal{H}^{(t)}$, it
suffices to prove the claim when $\eta $ is of the form%
\begin{equation*}
x\mapsto \sum_{ij=1}^{n}a_{ij}\left( x\right) ([\pi ^{(t)}]_{\sigma _{i}}\xi
^{(t)})_{x}\otimes ([\overline{\pi }^{(t)}]_{\sigma _{j}^{\prime }}\overline{%
\xi }^{(t)})_{x}\text{,}
\end{equation*}%
where $n\in \mathbb{N}$, $a_{ij}\in L^{\infty }\left( X\right) $, and $%
\sigma _{i},\sigma _{j}^{\prime }\in \left[ G\right] $ for $i,j=1,2,\ldots
,n $. For $z\in X$, fix $M(z)>0$ such that%
\begin{equation}
\max \left\{ \left\vert a_{ij}\left( r\left( \sigma _{k}^{-1}z\right)
\right) \right\vert \exp \left( -tM(z)\right) :1\leq i,j,k\leq n\right\}
\leq \frac{\varepsilon }{n^{4}}\text{.\label{Equation:bound}}  \tag{(i)}
\end{equation}%
By assumption for every non-null Borel subset $A$ of $X$ there exists a
non-null Borel subset $B$ of $A$ such that, for every $x\in B$, $\psi
|_{AHx} $ is unbounded. This easily implies that there exists $\rho \in %
\left[ H\right] $ such that, for a.e.\ $z,w\in X$, for every $1\leq i,k\leq n
$,%
\begin{equation*}
\psi \left( z\rho w\right) \geq M(z)^{1/2}+\psi \left( \sigma _{k}z\right)
^{1/2}+\psi \left( w\sigma _{i}^{\prime }\right) ^{1/2}\text{.}
\end{equation*}%
We have that%
\begin{eqnarray*}
&&\frac{1}{2}\left\vert \left\Vert ([\pi ]_{\rho }\otimes \lbrack \overline{%
\pi }]_{\rho })\eta -\eta \right\Vert ^{2}-2\right\vert \\
&=&\mathrm{Re}\sum_{i,j,k,l=1}^{n}\int_{x\in X}\overline{a}%
_{kl}(x)a_{ij}(x)\left\langle \xi _{x}^{(t)}\otimes \overline{\xi }%
_{x}^{(t)},\pi _{x\sigma _{k}^{-1}\rho \sigma _{i}^{\prime }}^{(t)}\xi
_{s(x\sigma _{k}^{-1}\rho \sigma _{i}^{\prime })}^{(t)}\otimes \overline{\pi 
}_{x\sigma _{l}^{-1}\rho \sigma _{j}^{\prime }}^{(t)}\overline{\xi }%
_{s(x\sigma _{l}^{-1}\rho \sigma _{j}^{\prime })}^{(t)}\right\rangle d\mu (x)
\\
&=&\mathrm{Re}\sum_{i,j,k,l=1}^{n}\int_{x\in X}\overline{a}%
_{kl}(x)a_{ij}(x)\exp \left( -t\psi \left( x\sigma _{k}^{-1}\rho \sigma
_{i}^{\prime }\right) -t\psi \left( x\sigma _{l}^{-1}\rho \sigma
_{j}^{\prime }\right) \right) d\mu (x)\text{.}
\end{eqnarray*}

Now let $(\mathcal{H}^{\psi },b^{\psi },\pi ^{\psi })$ be a triple obtained
from $\psi $ as in Proposition \ref{Proposition:nt}. Thus we have that $%
\mathcal{H}^{\psi }$ is a real Hilbert bundle over $X$, $\pi ^{\psi }$ is a
representation of $G$ on $\mathcal{H}^{\psi }$, and $b^{\psi }$ is a cocycle
for $\pi ^{\psi }$ such that 
\begin{equation*}
\psi \left( g^{-1}h\right) =\left\Vert b_{g}^{\psi }-b_{h}^{\psi
}\right\Vert ^{2}
\end{equation*}%
for $g,h\in G$. Thus, for a.e.\ $x\in X$, by the choice of $\rho $,%
\begin{eqnarray*}
\psi \left( x\sigma _{k}^{-1}\rho \sigma _{i}^{\prime }\right) &=&\left\Vert
b_{r\left( \sigma _{k}x\right) \rho \sigma _{i}^{\prime }}-b_{\sigma
_{k}x}\right\Vert ^{2} \\
&=&\left\Vert b_{r\left( \sigma _{k}x\right) \rho }+\pi _{r\left( \sigma
_{k}x\right) \rho }b_{s\left( r\left( \sigma _{k}x\right) \rho \right)
\sigma _{i}^{\prime }}-b_{\sigma _{k}x}\right\Vert ^{2} \\
&\geq &\left( \left\Vert b_{r\left( \sigma _{k}x\right) \rho }\right\Vert
-\left\Vert b_{s\left( r\left( \sigma _{k}x\right) \rho \right) \sigma
_{i}^{\prime }}\right\Vert -\left\Vert b_{\sigma _{k}x}\right\Vert \right)
^{2} \\
&=&\left( \psi \left( r\left( \sigma _{k}x\right) \rho \right) ^{1/2}-\psi
\left( s(\rho )\sigma _{i}^{\prime }\right) ^{1/2}-\psi \left( \sigma
_{k}x\right) ^{1/2}\right) ^{2} \\
&\geq &M\left( r\left( \sigma _{k}x\right) \right) \text{,}
\end{eqnarray*}%
where we used the choice of $\rho $ at the last step. Similarly,%
\begin{equation*}
\psi \left( \gamma _{l}^{-1}\rho \gamma _{j}^{\prime }\right) \geq M\left(
r\left( \sigma _{l}x\right) \right) \text{.}
\end{equation*}%
Hence we have that%
\begin{eqnarray*}
&&\left\vert \left\Vert (\left[ \pi \right] _{\rho }\otimes \left[ \overline{%
\pi }\right] _{\rho })\eta -\eta \right\Vert ^{2}-2\right\vert \\
&=&\mathrm{Re}\sum_{i,j,k,l=1}^{n}\int_{x\in X}\overline{a}%
_{kl}(x)a_{ij}(x)\exp \left( -t\psi \left( x\sigma _{k}^{-1}\rho \sigma
_{i}^{\prime }\right) -t\psi \left( x\sigma _{l}^{-1}\rho \sigma
_{j}^{\prime }\right) \right) d\mu (x) \\
&\leq &\mathrm{Re}\sum_{i,j,k,l=1}^{n}\int_{x\in X}\overline{a}%
_{kl}(x)a_{ij}(x)\exp \left( -tM\left( r\left( \sigma _{k}x\right) \right)
-tM\left( r\left( \sigma _{l}x\right) \right) \right) d\mu (x)\leq
\varepsilon
\end{eqnarray*}%
by \eqref{Equation:bound}. This concludes the proof.

(1)$\Leftrightarrow $(4) Suppose that $B$ is a non-null Borel subset of $X$,
and $\pi $ is a representation of $G$. Since $H$ is ergodic, by Remark \ref%
{Remark:extend} the equivalence (1)$\Leftrightarrow $(4) follows from the
equivalence (1)$\Leftrightarrow $(2) after replacing $G$ with $G|_{B}$.

(4)$\Leftrightarrow $(5) This follows from Theorem \ref%
{Theorem:ad-coboundaries}.
\end{proof}

\begin{lemma}
\label{Lemma:projection-invariant}Let $\pi $ be a representation of a
discrete pmp groupoid $G$ on $\mathcal{H}$ and fix $\delta >0$. Suppose that 
$\xi $ is a unit section for $\mathcal{H}$. Assume that for every $\gamma
\in G$ one has that $\left\Vert \xi _{r(\gamma )}-\pi _{\gamma }\xi
_{s(\gamma )}\right\Vert \leq \delta $. Then there exists an invariant
section $\eta $ of $\mathcal{H}$ such that$\left\Vert \eta _{x}-\xi
_{x}\right\Vert \leq \delta $ for a.e.\ $x\in X$.
\end{lemma}

\begin{proof}
For every $x\in G^{0}$ let $C_{x}$ be the closed convex hull of $\{\pi
_{\gamma ^{-1}}\xi _{r(\gamma )}:\gamma \in Gx\}$. Let then $\eta _{x}$ be
the (unique) element of minimal norm of $C_{x}$ for $x\in G^{0}$. If $x\in X$
then we have that, for every $\gamma _{1},\ldots ,\gamma _{n}\in Gx$ and $%
s_{1},\ldots ,s_{n}\in \left[ 0,1\right] $ such that $s_{1}+\cdots +s_{n}=1$
one has that%
\begin{equation*}
\left\Vert \sum_{i}s_{i}\pi _{\gamma _{i}^{-1}}\xi _{r\left( \gamma
_{i}\right) }-\xi _{x}\right\Vert \leq \sum_{i}s_{i}\left\Vert \xi _{r\left(
\gamma _{i}\right) }-\pi _{\gamma _{i}}\xi _{x}\right\Vert \leq \delta \text{%
.}
\end{equation*}%
Therefore 
\begin{equation*}
\left\Vert \zeta -\xi _{x}\right\Vert \leq \delta
\end{equation*}%
for every $\zeta \in C_{x}$ and in particular%
\begin{equation*}
\left\Vert \eta _{x}-\xi _{x}\right\Vert \leq \delta \text{.}
\end{equation*}%
By uniqueness of the element of least norm in $C_{x}$ one also has that $%
\eta $ is invariant.
\end{proof}

The proof of the following result is inspired by \cite[Theorem 1.2]%
{jolissaint_property_2005} and \cite[Theorem 4.8 and Theorem 4.12]%
{ad_cohomology_2005}.

\begin{theorem}
\label{Theorem:T}Let $G$ be a discrete pmp groupoid, and $K\leq H\leq G$ be
subgroupoids. Fix a countable subgroup $\Gamma $ of $\left[ G\right] $ that
covers $G$. Let $\left( X,\mu \right) $ be the common unit space of $K,H,G$.
Suppose that $H$ is ergodic. The following statements are equivalent:

\begin{enumerate}
\item There exists a finite subset $F$ of $\left[ G\right] $ and $\delta >0$
such that, whenever a representation $\pi $ of $G$ has an $\left( F,\delta
\right) $-invariant $K$-invariant unit section, then $\pi $ has an $H$%
-invariant unit section;

\item There exists a finite subset $F$ of $\left[ G\right] $ and $\delta >0$
such that, whenever a representation $\pi $ of $G$ on a Hilbert bundle $%
\mathcal{H}$ has an $\left( F,\delta \right) $-invariant $K$-invariant unit
section, then $\mathcal{H}$ contains a finite-dimensional $\pi |_{H}$%
-invariant sub-bundle;

\item For every complex-valued $K$-invariant Borel function $\psi $ on $G$
which is conditionally of negative type, there exists a non-null Borel set $%
A $ of $X$ such that, for every $x\in A$, $\psi |_{AHx}$ is bounded;

\item For any representation $\pi $ of $G$, one has that $H_{:K,H}^{1}\left(
\pi \right) $ is trivial.

\item the triple $K\leq H\leq G$ has property (T);

\item For every $\varepsilon >0$ and non-null Borel subset $B$ of $X$ there
exists a finite subset $F$ of $\Gamma $ and $\delta >0$ such that for every
normalized $K$-invariant function of positive type $\varphi $ on $G$ such
that $\max_{t\in F}\int_{x\in X}\left\vert \varphi \left( xt\right)
-1\right\vert ^{2}d\mu (x)\leq \delta $, there is a non-null Borel subset $A$
of $B$ such that $\mathrm{Re}\left( 1-\varphi (\gamma )\right) \leq
\varepsilon $ for a.e.\ $\gamma \in AHA$;

\item For every $\varepsilon >0$ and non-null Borel subset $B$ of $X$ there
exists a finite subset $F$ of $\Gamma $ and $\delta >0$ such that, if $\pi $
is a representation of $G$ on a Hilbert bundle $\mathcal{H}$, and $\xi $ is
a $\left( F,\delta \right) $-invariant $K$-invariant unit section for $%
\mathcal{H}$, then there is a non-null Borel subset $A$ of $B$ and an $H$%
-invariant section $\eta $ for $\pi $ such that $\left\Vert \xi _{x}-\eta
_{x}\right\Vert \leq \varepsilon $ for a.e.\ $x\in A$;

\item For every $\varepsilon >0$, there exists a finite subset $F$ of $\left[
G\right] $ and $\delta >0$ such that for every normalized $K$-invariant
function of positive type $\varphi $ on $G$ such that $\max_{t\in
F}\int_{x\in X}\left\vert \varphi \left( xt\right) -1\right\vert ^{2}d\mu
(x)\leq \delta $, there is a non-null Borel subset $A$ of $X$ such that $%
\mathrm{Re}\left( 1-\varphi (\gamma )\right) \leq \varepsilon $ for a.e.\ $%
\gamma \in AHA$;

\item For every $\varepsilon >0$, there exists a finite subset $F$ of $\left[
G\right] $ and $\delta >0$ such that, if $\pi $ is a representation of $G$
on a Hilbert bundle $\mathcal{H}$, and $\xi $ is a $\left( F,\delta \right) $%
-invariant $K$-invariant unit section for $\mathcal{H}$, then there is a
non-null Borel subset $A$ of $X$ and an $H$-invariant section $\eta $ for $%
\pi $ such that $\left\Vert \xi _{x}-\eta _{x}\right\Vert \leq \varepsilon $
for a.e.\ $x\in A$.
\end{enumerate}
\end{theorem}

\begin{proof}
Fix an increasing sequence $\left( F_{n}\right) $ of finite subsets of $%
\left[ G\right] $ whose union is $\Gamma $.

(1)$\Rightarrow $(2) Obvious.

(2)$\Rightarrow $(3) This is a consequence of Lemma \ref{Lemma:bounded}.

(3)$\Rightarrow $(4) As in the proof of \cite[Proposition 4.13]%
{ad_cohomology_2005}, it is enough to consider the case when $\pi $ is a
representation of $G$ on a bundle of real\emph{\ }Hilbert spaces; see also 
\cite[Lemma 4.11]{ad_cohomology_2005} and \cite[page 49]%
{harpe_propriete_1989}. Suppose that $b$ is a $K$-trivial cocycle for $\pi $%
. Define the $K$-invariant function of conditional negative type $\psi
:G\rightarrow \mathbb{R}$ by $\psi (\gamma )=\left\Vert b_{\gamma
}\right\Vert ^{2}$. Then by assumption, there exists a non-null Borel subset 
$A$ of $X$ such that, for every $x\in A$, $\psi |_{AHx}$ is bounded. This
implies by Theorem \ref{Theorem:ad-coboundaries} that the restriction of $b$
to $H$ is a coboundary for $\pi |_{H}$. Thus $H_{:K,H}^{1}\left( \pi \right) 
$ is trivial.

(4)$\Rightarrow $(5) Suppose by contradiction that there exists a
representation $\pi $ of $G$ that has almost invariant $K$-invariant unit
sections but it does not have an $H$-invariant unit section. The hypothesis
implies that $B_{:K}^{1}\left( \pi |_{H}\right) =Z_{:K,H}^{1}\left( \pi
\right) $. In particular, $B_{:K}^{1}\left( \pi |_{H}\right) $ is a closed
subspace of $S\left( H,\mathcal{H}\right) $. Let $S_{:K}\left( X,\mathcal{H}%
\right) $ be the space of $K$-invariant unit sections for $\mathcal{H}$,
which is a closed subspace of the space $S\left( X,\mathcal{H}\right) $ of
sections for $\mathcal{H}$. Define a map $\beta $ from the space $%
S_{:K}\left( X,\mathcal{H}\right) $ to $B_{:K}^{1}\left( \pi |_{H}\right) $
by%
\begin{equation*}
\beta \left( \xi \right) _{\gamma }:=\xi _{r(\gamma )}-\pi _{\gamma }\xi
_{s(\gamma )}
\end{equation*}%
for $\xi \in S\left( X,\mathcal{H}\right) $ and $\gamma \in H$. This map is
a continuous linear map from $S_{:K}\left( X,\mathcal{H}\right) $ onto $%
B_{:K}^{1}\left( \pi |_{H}\right) $. Since by assumption $\pi $ does not
have an $H$-invariant unit section, such a map is injective. Since a
continuous linear isomorphism between metrizable complete topological vector
spaces is a homeomorphism, $\beta $ is a homeomorphism. Since by assumption $%
\pi $ has almost invariant $K$-invariant unit sections, we can find a
sequence $\left( \xi ^{(n)}\right) $ of $K$-invariant unit sections in $%
S_{:K}\left( X,\mathcal{H}\right) $ such that $\beta (\xi ^{(n)})\rightarrow
0$. Therefore $\xi ^{(n)}\rightarrow 0$, contradicting the fact that the $%
\xi ^{(n)}$'s are unit sections.

(5)$\Rightarrow $(1) Assume that (1) does not hold. Then for every $n\in 
\mathbb{N}$ there exists a representation $\pi ^{(n)}$ on $\mathcal{H}^{(n)}$
without $H$-invariant unit sections which has a $\left( F_{n},2^{-n}\right) $%
-invariant $K$-invariant unit section $\xi ^{(n)}$. One can then consider
the direct sum $\pi $ of $\pi ^{(n)}$ for $n\in \mathbb{N}$. Then $\pi $ has
almost invariant $K$-invariant unit sections. Hence by assumption it has an $%
H$-invariant unit section $\xi $. One can write $\xi $ as the direct sum of
sections $\xi ^{(n)}$ for $\mathcal{H}^{(n)}$ for $n\in \mathbb{N}$. Since $%
\xi $ is a $H$-invariant, one has that $\xi ^{(n)}$ is $H$-invariant for
every $n\in \mathbb{N}$. Since $\xi $ is a unit section, there exists $n\in 
\mathbb{N}$ such that $\xi ^{(n)}$ is not identically zero. Since $H$ is
ergodic, this contradicts the assumption that $\pi ^{(n)}$ does not have $H$%
-invariant unit sections.

(3)$\Rightarrow $(6) Suppose that (6) fails. Then there exists $c>0$ and a
non-null Borel subset $B$ of $X$ such that for every $n\in \mathbb{N}$ one
can find a $K$-invariant function of positive type $\varphi $ on $G$ such
that $\max_{t\in F_{n}}\int_{x\in X}\left\vert \varphi \left( xt\right)
-1\right\vert d\mu (x)\leq 2^{-n}$ and for every non-null Borel subset $A$
of $B$ the set of $\gamma \in AHA$ such that $\mathrm{Re}\left( 1-\varphi
(\gamma )\right) \geq c$ is non-null. This allows one to find a sequence $%
\left( \varphi _{n}\right) $ of $K$-invariant functions of positive type on $%
G$ and Borel subsets $X_{n}$ of $X$ such that $\mu \left( X_{n}\right) \geq
1-2^{-n}$, $\left\vert \varphi \left( xt\right) -1\right\vert \leq 2^{-n}$
for $x\in X_{n}$ and $t\in F_{n}$, and such that for every non-null Borel
subset $A$ of $B$ the set of $\gamma \in AHA$ such that \textrm{Re}$\left(
1-\varphi (\gamma )\right) \geq c$ is non-null. One can then define, for
a.e.\ $\gamma \in G$,%
\begin{equation*}
\psi (\gamma )=\sum_{n\in \mathbb{N}}2^{n}\mathrm{Re}\left( 1-\varphi
_{n}(\gamma )\right) \text{.}
\end{equation*}%
This gives a $K$-invariant function of conditionally negative type on $G$
such that $\psi |_{AHA}$ is unbounded for every non-null Borel subset $A$ of 
$B$. By Lemma \ref{Lemma:bounded}, this implies that, for every non-null
Borel subset $A$ of $X$, for a.e.\ $x\in A$, $\psi |_{AHx}$ is unbounded.
Thus $\psi $ contradicts (3).

(6)$\Rightarrow $(7) Fix $\varepsilon >0$ and a non-null Borel subset $B$ of 
$X$. By assumption there exist a finite subset $F$ of $\Gamma $ and $\delta
>0$ such that, for every $K$-invariant normalized function of positive type $%
\varphi $ on $G$ such that $\max_{t\in F}\int_{x\in X}\left\vert \varphi
\left( xt\right) -1\right\vert d\mu (x)\leq \delta $, there is a non-null
Borel subset $A$ of $B$ such that $\mathrm{Re}\left( 1-\varphi (\gamma
)\right) \leq \varepsilon $ for every $\gamma \in AHA$. Suppose that $\pi $
is a representation of $G$ on $\mathcal{H}$ that has a $K$-invariant unit
section $\xi $ satisfying $\left\Vert \left[ \pi \right] _{t}\xi -\xi
\right\Vert \leq \delta $ for $t\in F$. Define the $K$-invariant normalized
function of positive type $\varphi $ on $G$ by $\varphi (\gamma
)=\left\langle \xi _{r(\gamma )},\pi _{\gamma }\xi _{s(\gamma
)}\right\rangle $. Then we have that, for $t\in F$,%
\begin{eqnarray*}
\int_{x\in X}\left\vert \varphi \left( xt\right) -1\right\vert d\mu \left(
x\right) &=&\int_{x\in X}\left\vert \left\langle \xi _{x},\pi _{xt}\xi
_{s\left( xt\right) }\right\rangle -1\right\vert d\mu (x)=\int_{x\in
X}\left\vert \left\langle \xi _{x},\pi _{xt}\xi _{s\left( xt\right) }-\xi
_{x}\right\rangle \right\vert d\mu (x) \\
&=&\left\vert \left\langle \xi ,\left[ \pi \right] _{t}\xi -\xi
\right\rangle \right\vert \leq \left\Vert \left[ \pi \right] _{t}\xi -\xi
\right\Vert \leq \delta \text{.}
\end{eqnarray*}%
Therefore by assumption there exists a non-null Borel subset $A$ of $B$ such
that, for $\gamma \in AHA$, one has that $\mathrm{Re}\left( 1-\varphi
(\gamma )\right) \leq \varepsilon $. Therefore we have that, for $\gamma \in
AHA$, 
\begin{equation*}
\frac{1}{2}\left\Vert \pi _{\gamma }\xi _{s(\gamma )}-\xi _{r(\gamma
)}\right\Vert ^{2}=\mathrm{Re}\left( 1-\varphi (\gamma )\right) \leq
\varepsilon \text{.}
\end{equation*}%
Therefore by Lemma \ref{Lemma:projection-invariant} applied to the
representation $\pi _{A}$ of $H|_{A}$ on $\mathcal{H}|_{A}$ obtained from $%
\pi $ by restriction, we have that there exists a unit section $\eta $ for $%
\mathcal{H}|_{A}$ which is invariant for $\pi _{A}$ and such that $%
\left\Vert \xi _{x}-\eta _{x}\right\Vert \leq \varepsilon $ for $x\in A$.
Since $H$ is ergodic, this concludes the proof by Remark \ref{Remark:extend}.

(8)$\Rightarrow $(9) This is the same as (6)$\Rightarrow $(7).

Finally the implications (7)$\Rightarrow $(1), (6)$\Rightarrow $(8), and (9)$%
\Rightarrow $(1) are obvious.
\end{proof}

\begin{remark}
Theorem \ref{Theorem:T} in the case when $G$ is a countable discrete group
and $K$ is the trivial subgroup recovers \cite[Theorem 1]%
{jolissaint_property_2005}. Theorem \ref{Theorem:T} recovers \cite[Theorems
4.8, Theorem 4.12 and Theorem 5.22]{ad_cohomology_2005} in the case when $%
H=G $ and $K$ is the trivial subgroupoid.
\end{remark}

\subsection{Property (T) for action groupoids}

Suppose that $\left( X,\mu \right) $ is a standard probability space. A
standard probability space fibered over $\left( X,\mu \right) $ is a triple $%
\left( Y,\nu ,p\right) $ where $\left( Y,\nu \right) $ is a standard
probability space, and $p:Z\rightarrow X$ is a Borel map such that $p_{\ast
}\nu =\mu $. We also write $\left( Y,\nu \right) =\bigsqcup_{x\in X}\left(
Y_{x},\nu _{x}\right) $ where $\left( \nu _{x}\right) _{x\in X}$ is the
disintegration of $\nu $ with respect to $\mu $. One can consider the space $%
\mathrm{\mathrm{Aut}}\left( \bigsqcup_{x\in X}Y_{x}\right) $ of Borel maps $%
T:Y_{x}\rightarrow Y_{y}$ for $x,y\in X$ such that $T_{\ast }\nu _{x}=\nu
_{y}$. One can define a standard Borel structure on $\mathrm{\mathrm{Aut}}%
\left( \bigsqcup_{x\in X}Y_{x}\right) $ that turns it into a standard Borel
groupoid, whose unit space can be identified with $X$.

Suppose that $G$ is a discrete pmp groupoid, and $\bigsqcup_{x\in X}Y_{x}$
is a standard probability space fibered over $G^{0}$. A
probability-measure-preserving (pmp) \emph{action }$\theta $ of $G$ on $%
\bigsqcup_{x\in X}Y_{x}$ is a homomorphism $\gamma \mapsto \theta _{\gamma }$
from $G$ to $\mathrm{Aut}\left( \bigsqcup_{x\in X}Y_{x}\right) $ that is the
identity on the unit space. We set $\gamma \cdot ^{\theta }y=\theta _{\gamma
}(y)$ for $\gamma \in G$ and $y\in Y_{s(\gamma )}$. One can then define the 
\emph{transformation groupoid }$G\ltimes ^{\theta }Y$. This is the set of
pairs $\left( \gamma ,y\right) $ such that $\gamma \in G$ and $y\in
Y_{s(\gamma )}$, which is a Borel subset of $G\times Y$ endowed with the
product topology. Identifying an element $y$ of $Y_{x}$ for $x\in X$ with
the pair $\left( x,y\right) $, one can identify $Y$ with the unit space of $%
G\ltimes ^{\theta }Y$. The source and range maps on $G\ltimes ^{\theta }Y$
are defined by $s\left( \gamma ,y\right) =y$ and $r\left( \gamma ,y\right)
=\gamma \cdot ^{\theta }y$. Composition of arrows is given by $\left( \gamma
,y\right) \left( \gamma ^{\prime },y^{\prime }\right) =\left( \gamma \gamma
^{\prime },y^{\prime }\right) $ whenever $\gamma ^{\prime }\cdot ^{\theta
}y^{\prime }=y$.

Suppose that $G$ is a discrete pmp groupoid, and $\theta $ is an action of $%
G $ on the standard probability space $Y=\bigsqcup_{x\in G^{0}}Y_{x}$. One
can then consider the transformation groupoid $G\ltimes ^{\theta }Y$.
Suppose that $K\leq H\leq G$ are subgroupoids. A representation $\pi $ of $G$
on $\mathcal{H}$ induces a representation $\pi \ltimes ^{\theta }Y$ of $%
G\ltimes ^{\theta }Y$ on $\mathcal{H}$ defined by $\left( \pi \ltimes
^{\theta }Y\right) _{\gamma \ltimes ^{\theta }x}\left( \xi \right) =\pi
_{\gamma }\left( \xi \right) $ for $\xi \in \mathcal{H}_{x}$. Let $b\in
S\left( G,\mathcal{H}\right) $ be a cocycle for $\pi $. Then one can define
a cocycle $b\ltimes ^{\theta }Y$ for $\pi \ltimes ^{\theta }Y$ by setting $%
\left( b\ltimes ^{\theta }Y\right) _{\gamma \ltimes ^{\theta }x}=b_{\gamma }$%
. It is clear that if $b$ is $K$-trivial then $b\ltimes ^{\theta }Y$ is $%
K\ltimes ^{\theta }Y$-trivial. Furthermore if $b|_{H}$ is a coboundary for $%
\pi |_{H}$, then $b|_{H}\ltimes ^{\theta }Y$ is a coboundary for $\left( \pi
\ltimes ^{\theta }Y\right) |_{H\ltimes ^{\theta }Y}$. Therefore the
assignment $b\mapsto b\ltimes ^{\theta }Y$ defines a homomorphism $%
H_{:K,H}^{1}\left( \pi \right) \rightarrow H_{:K\ltimes ^{\theta }Y,H\ltimes
^{\theta }Y}^{1}\left( \pi \ltimes Y\right) $. The following lemma is an
immediate consequence of \cite[Lemma 5.12]{ad_cohomology_2005}.

\begin{lemma}
\label{Lemma:T-action}Adopting the notation above, suppose that $H$ is
ergodic and the action $\theta |_{H}$ of $H$ on $Y$ is ergodic. Then the
homomorphism $H_{:K,H}^{1}\left( \pi \right) \rightarrow H_{:K\ltimes
^{\theta }Y,H\ltimes ^{\theta }Y}^{1}\left( \pi \ltimes Y\right) $ is
injective.
\end{lemma}

\begin{proof}
Suppose that $b$ is a $K$-invariant cocycle for $\pi $ on $\mathcal{H}$.
Assume that $\left( b\ltimes ^{\theta }Y\right) |_{H\ltimes ^{\theta }Y}$ is
a coboundary. Then by \cite[Lemma 5.12]{ad_cohomology_2005}, $b|_{H}$ is a
coboundary. Thus the map $H_{:K,H}^{1}\left( \pi \right) \rightarrow
H_{:K\ltimes ^{\theta }Y,H\ltimes ^{\theta }Y}^{1}\left( \pi \ltimes
Y\right) $ is injective.
\end{proof}

Suppose now, adopting the notation above, that $\pi $ is a representation of 
$H\ltimes ^{\theta }Y$ on $\mathcal{H}$. One can then define the \emph{%
induced representation} $\hat{\pi}$ of $G$ as follows. Consider the Hilbert
bundle $\widehat{\mathcal{H}}=\bigsqcup_{x\in X}L^{2}\left( Y_{x}\right)
\otimes \mathcal{H}_{x}\cong \bigsqcup_{x\in X}L^{2}\left( Y_{x},\mathcal{H}%
_{x}\right) $. Then the representation $\hat{\pi}$ on $\widehat{\mathcal{H}}$
is defined by setting, for $\gamma \in G$ and $\xi \in L^{2}\left(
Y_{s(\gamma )},\mathcal{H}_{s(\gamma )}\right) $, $\hat{\pi}_{\gamma }\xi $
to be the element of $L^{2}\left( Y_{r(\gamma )},\mathcal{H}_{r(\gamma
)}\right) $ given by%
\begin{equation*}
(\hat{\pi}_{\gamma }\xi )(y)=\pi _{\gamma \ltimes ^{\theta }\gamma
^{-1}y}\xi (\gamma ^{-1}\cdot ^{\theta }y)
\end{equation*}%
for $y\in Y_{r(\gamma )}$. Observe that this is indeed a representation. In
fact we have that%
\begin{eqnarray*}
\hat{\pi}_{\gamma }(\hat{\pi}_{\rho }\xi ) &=&\pi _{\gamma \ltimes ^{\theta
}\gamma ^{-1}y}(\hat{\pi}_{\rho }\xi )(\gamma ^{-1}\cdot ^{\theta }y)=\pi
_{\gamma \ltimes ^{\theta }\gamma ^{-1}y}(\pi _{\rho \rtimes ^{\theta }\rho
^{-1}\gamma ^{-1}y})\xi (\rho ^{-1}\gamma ^{-1}\cdot ^{\theta }y) \\
&=&\pi _{\gamma \rho \ltimes ^{\theta }\left( \gamma \rho \right) ^{-1}y}\xi
(\left( \gamma \rho \right) ^{-1}\cdot ^{\theta }y)=\hat{\pi}_{\gamma \rho
}(y)\text{.}
\end{eqnarray*}

Given a section $\xi $ for $\mathcal{H}$ one can define the section $\hat{\xi%
}$ of $\widehat{\mathcal{H}}$ by setting $\hat{\xi}_{x}=\xi |_{Y_{x}}\in
L^{2}\left( Y_{x},\mathcal{H}_{x}\right) $ for $x\in G^{0}$. It is clear
that if $\xi $ is $K\ltimes Y$-invariant, then $\hat{\xi}$ is $K$-invariant.
Furthermore it is shown in \cite[Section 5]{ad_cohomology_2005} that if $%
\left( \xi _{n}\right) $ is a sequence of almost $\pi $-invariant unit
sections for $\mathcal{H}$, then $(\hat{\xi}_{n})$ is a sequence of almost $%
\hat{\pi}$-invariant unit sections for $\widehat{\mathcal{H}}$. As in the
proof of \cite[Theorem 5.15]{ad_cohomology_2005} one can deduce from Lemma %
\ref{Lemma:T-action} and these observations the following.

\begin{theorem}
\label{Theorem:action-groupoid-T}Suppose that $G$ is a discrete pmp
groupoid, and $K\leq H\leq G$ are subgroupods such that $H$ is ergodic. Let $%
\theta $ be an action of $G$ on a standard probability space $Y$ such that $%
\theta |_{H}$ is ergodic. Then $K\leq H\leq G$ has property (T) if and only
if $K\ltimes ^{\theta }Y\leq H\ltimes ^{\theta }Y\leq G\ltimes ^{\theta }Y$
has property (T).
\end{theorem}

\section{Rigid inclusions of von Neumann algebras\label{Section:rigid}}

Suppose that $G$ is a discrete pmp groupoid. One can then consider the
Hilbert bundle $\mathcal{H}=\bigsqcup_{x\in G^{0}}\ell ^{2}\left( xG\right) $%
. Observe that one can canonically identify $L^{2}(G^{0},\mathcal{H})$ with $%
L^{2}(G)$. The $\emph{left}$ \emph{regular representation }of $G$ is the
representation $\lambda $ of $G$ on $\mathcal{H}$ defined as follows. For $%
\gamma \in G$ let $\delta _{\gamma }\in \ell ^{2}\left( r(\gamma )G\right) $
be the indicator function of $\left\{ \gamma \right\} \subset r(\gamma )G$.\
Then $\lambda _{\rho }\delta _{\gamma }=\delta _{\rho \gamma }$ for $\left(
\rho ,\gamma \right) \in G^{2}$. This gives rise the representation $\left[ %
\left[ \lambda \right] \right] $ of $\left[ \left[ G\right] \right] $ on $%
L(G)$. The \emph{groupoid von Neumann algebra} of $G$ is defined to be the
von Neumann algebra $L(G)\subset B\left( L^{2}(G)\right) $ generated by the
elements $u_{\sigma }:=\left[ \left[ \lambda \right] \right] _{\sigma }$ for 
$\sigma \in \left[ \left[ G\right] \right] $. The main goal of this section
is to provide a characterization of property (T) for (triples of) groupoids
in terms of the associated groupoid von Neumann algebra.

\subsection{Hilbert bimodules and ucp maps}

Suppose that $\left( M,\tau \right) $ is a tracial von Neumann algebra. We
let $L^{2}\left( M\right) $ be the Hilbert space obtained from $\left(
M,\tau \right) $ via the GNS construction, and $M\rightarrow L^{2}\left(
M\right) $, $x\mapsto \left\vert x\right\rangle $ be the canonical
inclusion. Thus $\left\vert 1\right\rangle $ is the canonical cyclic vector
of $M$.

A (Hilbert) $M$-\emph{bimodule} is a Hilbert space $\mathfrak{H}$ endowed
with commuting normal *-representations $\pi $ of $M$ and $\rho $ of $M^{%
\mathrm{op}}$ on $\mathfrak{H}$. In this case, given $x,y\in M$ and $\xi \in 
\mathfrak{H}$, one writes $x\xi y$ for $\pi (x)\rho \left( y^{\mathrm{op}%
}\right) \xi $. A vector $\xi $ of $\mathfrak{H}$ is called \emph{tracial }%
if it satisfies $\left\langle \xi ,a\xi \right\rangle =\tau (a)$ and $%
\left\langle \xi ,\xi b\right\rangle =\tau (b)$ for every $a\in M$ and $b\in
N$. Given a subset $F$ of $M$ and $\varepsilon >0$, a vector $\xi $ in $%
\mathfrak{H}$ is $F$-central if it satisfies $a\xi =\xi a$ for $a\in F$, and 
$\left( F,\varepsilon \right) $-central if it satisfies $\left\Vert a\xi
-\xi a\right\Vert \leq \varepsilon $ for $a\in F$. The \emph{adjoint} $M$%
-bimodule $\overline{\mathfrak{H}}$ is equal to the conjugate Hilbert space
of $\mathfrak{H}$ endowed with the bimodule structure given by $x\overline{%
\xi }y=\overline{y^{\ast }\xi x^{\ast }}$ for $x,y\in M$ and $\xi \in 
\mathfrak{H}$.

A linear map $\phi :M\rightarrow M$ is \emph{completely positive} (cp) if,
for every $n\in \mathbb{N}$, $\mathrm{id}_{M_{n}\left( \mathbb{C}\right)
}\otimes \phi :M_{n}\left( \mathbb{C}\right) \otimes M\rightarrow
M_{n}\left( \mathbb{C}\right) \otimes N$ maps positive elements to positive
elements. If furthermore $\phi \left( 1\right) =1$, then $\phi $ is \emph{%
unital} completely positive (ucp). A map $\phi :M\rightarrow M$ is \emph{%
trace-preserving} if $\tau \circ \phi =\tau $. If $A$ is a subalgebra of $M$%
, then a map $\phi :M\rightarrow M$ is an $A$-bimodule map if it satisfies $%
\phi \left( ax\right) =a\phi (x)$ and $\phi \left( xa\right) =\phi (x)a$ for 
$x\in M$ and $a\in A$. Suppose that $\phi :M\rightarrow M$ is a cp $A$%
-bimodule map satisfying $\tau \circ \phi \leq \tau $ and $\phi \left(
1\right) \leq 1$. Setting $T_{\phi }\left\vert x\right\rangle =\left\vert
\phi (x)\right\rangle $ for $x\in M$ defines a bounded operator $T_{\phi }$
on $L^{2}\left( M\right) $. The adjoint $T_{\phi }^{\ast }$ of $T_{\phi }$
is of the form $T_{\phi ^{\ast }}$, where $\phi ^{\ast }:M\rightarrow M$ is
a cp $A$-bimodule map satisfying $\tau \circ \phi ^{\ast }\leq \tau $ and $%
\phi ^{\ast }\left( 1\right) \leq 1$; see \cite[Lemma 1.2.1]{popa_class_2006}%
.

Given a nonzero normal cp $A$-bimodule map $\phi :M\rightarrow M$ satisfying 
$\phi \left( 1\right) \leq 1$ and $\tau \circ \phi \leq \tau $, one can
define a Hilbert $M$-bimodule associated with $\phi $ as follows. Consider
the completion $\mathfrak{H}_{\phi }$ of $M\odot M$ with respect to the
inner product defined by%
\begin{equation*}
\left\langle a\otimes b,c\otimes d\right\rangle =\tau \left( b^{\ast }\phi
\left( a^{\ast }c\right) d\right) \text{.}
\end{equation*}%
The $M$-bimodule structure is induced by the maps%
\begin{equation*}
x\left( a\otimes b\right) y=xa\otimes by
\end{equation*}%
for $x,a\in M$ and $b,y\in N$. Denoting by $\xi _{\phi }$ the vector of $%
\mathfrak{H}_{\phi }$ obtained from $1\otimes 1$ one has that%
\begin{equation*}
\left\langle b\xi _{\phi }x,a\xi _{\phi }y\right\rangle =\left\langle
b\otimes x,a\otimes y\right\rangle =\tau \left( x^{\ast }\phi \left( b^{\ast
}a\right) y\right)
\end{equation*}%
In particular we have that $\left\langle \xi _{\phi },\xi _{\phi }\cdot
\right\rangle =\tau \left( \phi \left( 1\right) \cdot \right) \leq \tau $
and $\left\langle \xi _{\phi },\cdot \xi _{\phi }\right\rangle =\tau \circ
\phi \leq \tau $. Since $\phi $ is an $A$-bimodule map, $\xi _{\phi }$ is $A$%
-central. The vector $\xi _{\phi }$ is \emph{cyclic} for $\mathfrak{H}_{\phi
}$, in the sense that $\left\{ a\xi _{\phi }b:a,b\in M\right\} $ has dense
linear span in $\mathfrak{H}_{\phi }$. If $\phi $ is ucp and
trace-preserving, then $\xi _{\phi }$ is a tracial unit vector.

Conversely, suppose that $\mathfrak{H}$ is an $M$-bimodule with an $A$%
-central cyclic vector $\xi $ satisfying $\left\langle \xi ,\cdot \xi
\right\rangle \leq \tau $ and $\left\langle \xi ,\xi \cdot \right\rangle
\leq \tau $. One can define a normal cp $A$-bimodule map $\phi :M\rightarrow
M$ by setting $\phi (x)=L_{\xi }^{\ast }xL_{\xi }$. Here $L_{\xi
}:L^{2}(M)\rightarrow \mathfrak{H}$ is the operator defined by $L_{\xi
}\left\vert y\right\rangle =\xi y$ for $y\in M$. If $\xi $ is a tracial unit
vector, then $\phi $ is a ucp trace-preserving map. These constructions are
inverse of each other. More information on the correspondence between cp
maps and Hilbert bimodules can be found in \cite[Section 1]{popa_class_2006}.

\subsection{Completely positive maps and functions of positive type}

Let $G$ be a discrete pmp groupoid with unit space $X$. For a function of
positive type $\varphi $ on a discrete pmp groupoid $G$ and $\sigma \in %
\left[ \left[ G\right] \right] $, denote by $\varphi (\sigma )\in L^{\infty
}\left( X\right) $ the function 
\begin{equation*}
x\mapsto \left\{ 
\begin{array}{cc}
\varphi \left( x\sigma \right) & \text{if }x\in \mathrm{\mathrm{ran}}(\sigma
)\ \text{,} \\ 
0 & \text{otherwise.}%
\end{array}%
\right.
\end{equation*}%
The proof of the following proposition is similar to the proofs of \cite[%
Lemma 1.1]{haagerup_example_1978} and \cite[Proposition 3.5.4]%
{aaserud_applications_2011}. Recall that, if $\sigma \in \left[ \left[ G%
\right] \right] $, then we let $u_{\sigma }$ be the element $\left[ \left[
\lambda \right] \right] _{\sigma }$ of $L(G)\subset B\left( L^{2}(G)\right) $%
, where $\lambda $ is the left regular representation of $G$.

\begin{proposition}
\label{Proposition:ucp-groupoid}Suppose that $G$ is a discrete pmp group, $%
K\leq G$ is a subgroupoid, and $\varphi $ is a normalized $K$-invariant
function of positive type on $G$. Then there exists a unique
trace-preserving, $L(K)$-bimodule normal ucp map $\phi :L(G)\rightarrow L(G)$
such that%
\begin{equation*}
\phi (u_{\sigma })=\varphi (\sigma )u_{\sigma }
\end{equation*}%
for $\sigma \in \left[ G\right] $.
\end{proposition}

\begin{proof}
Consider the GNS representation $\left( \pi ^{\varphi },\mathcal{H}^{\varphi
},\xi ^{\varphi }\right) $ of $G$ associated with $\varphi $. Choose an
orthonormal basic sequence $\left( e^{(i)}\right) _{i\in \mathbb{N}}$ for $%
\mathcal{H}^{\varphi }$. Define $a^{(i)}\in L^{\infty }(G)$ by setting%
\begin{equation*}
a_{\gamma }^{(i)}:=\left\langle \xi _{r(\gamma )}^{\varphi },\pi _{\gamma
}^{\varphi }e_{s(\gamma )}^{(i)}\right\rangle =\left\langle \pi _{\gamma
}^{\varphi \ast }\xi _{r(\gamma )}^{\varphi },e_{s(\gamma
)}^{(i)}\right\rangle \text{.}
\end{equation*}%
For $\gamma ,\rho \in G$ one has that%
\begin{eqnarray*}
\sum_{i\in \mathbb{N}}\overline{a}_{\rho }^{(i)}a_{\gamma }^{(i)}
&=&\sum_{i\in \mathbb{N}}\left\langle e_{s(\rho )}^{(i)},\pi _{\rho
}^{\varphi \ast }\xi _{r(\rho )}^{\varphi }\right\rangle \left\langle \pi
_{\gamma }^{\varphi \ast }\xi _{r(\gamma )}^{\varphi },e_{s(\gamma
)}^{(i)}\right\rangle \\
&=&\sum_{i,j\in \mathbb{N}}\left\langle e_{s(\rho )}^{(i)},\pi _{\rho
}^{\varphi \ast }\xi _{r(\rho )}^{\varphi }\right\rangle \left\langle \pi
_{\gamma }^{\varphi \ast }\xi _{r(\gamma )}^{\varphi },e_{s(\gamma
)}^{(j)}\right\rangle \left\langle e_{s(\gamma )}^{(i)},e_{s(\gamma
)}^{(j)}\right\rangle \\
&=&\sum_{i,j\in \mathbb{N}}\left\langle \left\langle \pi _{\rho }^{\varphi
\ast }\xi _{r(\rho )}^{\varphi },e_{s(\rho )}^{(i)}\right\rangle e_{s(\gamma
)}^{(i)},\left\langle \pi _{\gamma }^{\varphi \ast }\xi _{r(\gamma
)}^{\varphi },e_{s(\gamma )}^{(j)}\right\rangle e_{s(\gamma
)}^{(j)}\right\rangle \\
&=&\left\langle \pi _{\rho }^{\varphi \ast }\xi _{r(\rho )}^{\varphi },\pi
_{\gamma }^{\varphi \ast }\xi _{r(\gamma )}^{\varphi }\right\rangle \\
&=&\left\langle \xi _{r(\gamma )}^{\varphi },\pi _{\rho \gamma
^{-1}}^{\varphi }\xi _{r(\gamma )}^{\varphi }\right\rangle \\
&=&\varphi \left( \rho \gamma ^{-1}\right) \text{.}
\end{eqnarray*}%
For $T\in L(G)\subset B\left( L^{2}(G)\right) $ set $\phi \left( T\right)
=\sum_{i\in \mathbb{N}}a^{(i)\ast }Ta^{(i)}$. The convergence is in strong
operator topology, since $\sum_{i=1}^{n}a_{i}^{\ast }a_{i}\leq 1$ for every $%
n\in \mathbb{N}$.

Now for $\xi \in L^{2}(G)$ and $i\in \mathbb{N}$ we have that%
\begin{eqnarray*}
(a^{(i)\ast }bu_{\sigma }a^{(i)}\xi )_{\gamma } &=&\overline{a}_{\gamma
}^{(i)\ast }(bu_{\sigma }a^{(i)}\xi )_{\gamma } \\
&=&\overline{a}_{\gamma }^{(i)\ast }b_{r(\gamma )}(u_{\sigma }a^{(i)}\xi
)_{\gamma } \\
&=&\overline{a}_{\gamma }^{(i)\ast }b_{r(\gamma )}(a^{(i)}\xi )_{\sigma
^{-1}\gamma } \\
&=&\overline{a}_{\gamma }^{(i)\ast }b_{r(\gamma )}a_{\sigma ^{-1}\gamma
}^{(i)}\xi _{\sigma ^{-1}\gamma }
\end{eqnarray*}%
if $r(\gamma )\in \sigma \sigma ^{-1}$, and $\left( a^{(i)\ast }bu_{\sigma
}a^{(i)}\xi \right) _{\gamma }=0$ otherwise. Therefore we have that%
\begin{eqnarray*}
\left( \phi \left( bu_{\sigma }\right) \xi \right) _{\gamma } &=&\sum_{i\in 
\mathbb{N}}((a^{(i)})^{\ast }bu_{\sigma }a^{(i)}\xi )_{\gamma } \\
&=&\sum_{i\in \mathbb{N}}((\overline{a}_{\gamma }^{(i)\ast }a_{\sigma
^{-1}\gamma }^{(i)})b_{r(\gamma )}\xi _{\sigma ^{-1}\gamma } \\
&=&\varphi (r(\gamma )\sigma )b_{r(\gamma )}\xi _{\sigma ^{-1}\gamma } \\
&=&(\varphi (\sigma )bu_{\sigma }\xi )_{\gamma }
\end{eqnarray*}%
if $r(\gamma )\in \sigma \sigma ^{-1}$, and%
\begin{equation*}
(\phi \left( bu_{\sigma }\right) \xi )_{\gamma }=0=(\varphi (\sigma
)bu_{\sigma }\xi )_{\gamma }
\end{equation*}%
otherwise. This shows that $\phi \left( bu_{\sigma }\right) =\varphi (\sigma
)bu_{\sigma }$. Since $\sum_{i\in \mathbb{N}}a_{i}^{\ast }a_{i}=1$, we have
that $\phi $ is unital. Clearly, $\phi $ is normal and completely positive,
and hence completely contractive. If $\sigma ^{\prime }\in \left[ K\right] $
then we have that, since $\varphi $ is $K$-invariant,%
\begin{equation*}
\phi \left( bu_{\sigma }u_{\sigma ^{\prime }}\right) =\phi \left( bu_{\sigma
\sigma ^{\prime }}\right) =\varphi (\sigma \sigma ^{\prime })bu_{\sigma
\sigma ^{\prime }}=\varphi (\sigma )bu_{\sigma }u_{\sigma ^{\prime }}=\phi
\left( bu_{\sigma }\right) u_{\sigma ^{\prime }}
\end{equation*}%
and%
\begin{equation*}
\phi \left( u_{\sigma ^{\prime }}bu_{\sigma }\right) =\phi \left( \theta
_{\sigma ^{\prime }}(b)u_{\sigma ^{\prime }\sigma }\right) =\varphi (\sigma
^{\prime }\sigma )\theta _{\sigma ^{\prime }}(b)u_{\sigma ^{\prime
}}u_{\sigma }=u_{\sigma ^{\prime }}\phi \left( bu_{\sigma }\right) \text{.}
\end{equation*}%
Similarly, if $a\in L^{\infty }\left( X\right) \subset L(G)$ one has that%
\begin{equation*}
\phi \left( abu_{\sigma }\right) =a\phi \left( bu_{\sigma }\right)
\end{equation*}%
and%
\begin{equation*}
\phi \left( bu_{\sigma }a\right) =\phi \left( bu_{\sigma }\right) a\text{.}
\end{equation*}%
These equations together with that fact that $\phi $ is a normal ucp map
imply that $\phi $ is an $L(K)$-bimodule map. This concludes the proof.
\end{proof}

\subsection{Rigidity for von Neumann algebras}

In order to characterize property (T) for triples of groupoids, we introduce
a notion of rigidity for a nested quadruple of von Neumann algebra.

\begin{definition}
\label{Definition:(T)-vN}Let $\left( M,\tau \right) $ be a tracial von
Neumann algebra and $Z\subset A\subset B\subset M$ be von Neumann
subalgebras. Then the quadruple $Z\subset A\subset B\subset M$ is \emph{rigid%
} if for every $\varepsilon >0$ and nonzero projection $p_{0}\in Z$ there is
a finite subset $F$ of $M$ and $\delta >0$ such that for any Hilbert $M$%
-bimodule $\mathfrak{H}$ with an $\left( F,\delta \right) $-central and $A$%
-central tracial unit vector $\xi \in \mathfrak{H}$ there is a nonzero
projection $p\in Z$ such that $p\leq p_{0}$ and for every projection $q\in Z$
such that $q\leq p$ one has that $\left\Vert x\xi -\xi x\right\Vert \leq
\tau \left( q\right) ^{1/2}\left\Vert x\right\Vert \varepsilon $ for $x\in
qBq$.
\end{definition}

As in the case of rigidity for pairs of von Neumann algebras as defined in 
\cite{popa_class_2006}, one can provide several equivalent characterizations
of rigidity for quadruples $Z\subset A\subset B\subset M$. The proof of this
fact is standard, and follows arguments from \cite%
{popa_class_2006,peterson_notion_2005,ioana_amalgamated_2008,peterson_cohomology_2009}%
. We present a full proof, for the reader's convenience.

\begin{proposition}
\label{Proposition:characterize-T-vN}Let $\left( M,\tau \right) $ be a
tracial von Neumann algebra and $Z\subset A\subset B\subset M$ be von
Neumann subalgebras. The following assertions are equivalent:

\begin{enumerate}
\item $Z\subset A\subset B\subset M$ is rigid;

\item for every $\varepsilon >0$ and nonzero projection $p_{0}\in Z$ there
exist a finite subset $F$ of $M$ and $\delta >0$ such that if $\phi
:M\rightarrow M$ is a normal trace-preserving ucp $A$-bimodule map such that 
$\max_{x\in F}\left\Vert \phi (x)-x\right\Vert _{2}\leq \delta $, then there
exists a nonzero projection $p\in Z$ such that $p\leq p_{0}$ and for every
projection $q\in Z$ such that $q\leq p$ one has that $\left\Vert \phi
(b)-b\right\Vert _{2}\leq \varepsilon \tau \left( q\right) ^{1/2}\left\Vert
b\right\Vert $ for every $b\in qBq$.

\item for every $\varepsilon >0$ and nonzero projection $p_{0}\in Z$ there
is a finite subset $F$ of $M$ and $\delta >0$ such that if $\phi
:M\rightarrow M$ is a normal cp $A$-bimodule map such that $\tau \circ \phi
\leq \tau $, $\phi \left( 1\right) \leq 1$, $\phi =\phi ^{\ast }$, and $%
\max_{x\in F}\left\Vert \phi (x)-x\right\Vert _{2}\leq \delta $, then there
exists a nonzero projection $p\in Z$ such that $p\leq p_{0}$ and for every
projection $q\in Z$ such that $q\leq p$ one has that $\left\Vert \phi
(b)-b\right\Vert _{2}\leq \varepsilon \tau \left( q\right) ^{1/2}\left\Vert
b\right\Vert $ for every $b\in qBq$;

\item for every $\varepsilon >0$ and nonzero projection $p_{0}\in Z$ there
is a finite subset $F$ of $M$ and $\delta >0$ such that if $\phi
:M\rightarrow M$ is a normal cp $A$-bimodule map such that $\tau \circ \phi
\leq \tau $, $\phi \left( 1\right) \leq 1$, and $\max_{x\in F}\left\Vert
\phi (x)-x\right\Vert _{2}\leq \delta $, then there exists a nonzero
projection $p\in Z$ such that $p\leq p_{0}$ and for every projection $q\in Z$
such that $q\leq p$ one has that $\left\Vert \phi (b)-b\right\Vert _{2}\leq
\varepsilon \tau \left( q\right) ^{1/2}\left\Vert b\right\Vert $ for every $%
b\in qBq$;

\item for every $\varepsilon >0$ and nonzero projection $p_{0}\in Z$ there
is a finite subset $F$ of $M$ and $\delta >0$ such that if $\mathfrak{H}$ is
an $M$-bimodule and $\xi \in \mathfrak{H}$ is an $\left( F,\delta \right) $%
-central and $A$-central unit vector satisfying $\left\langle \xi ,\cdot \xi
\right\rangle \leq \tau $ and $\left\langle \xi ,\xi \cdot \right\rangle
\leq \tau $, then there exists a nonzero projection $p\in Z$ such that $%
p\leq p_{0}$ and for every projection $q\in Z$ such that $q\leq p$ one has
that $\left\Vert x\xi -\xi x\right\Vert \leq \tau \left( q\right)
^{1/2}\left\Vert x\right\Vert \varepsilon $ for every $x\in qBq$.
\end{enumerate}
\end{proposition}

\begin{proof}
(1)$\Rightarrow $(2) Let $F$ be a finite subset of $M$, and $\delta >0$.
Suppose that $\phi :M\rightarrow M$ is a normal trace-preserving ucp $A$%
-bimodule map such that $\max_{x\in F}\left\Vert x\right\Vert _{2}\left\Vert
\phi (x)-x\right\Vert _{2}\leq \delta /2$. Let $(\mathfrak{H},\xi )$ be the
corresponding Hilbert $M$-bimodule with distinguished $A$-central tracial
unit vector $\xi $. For $x\in M$ we have that%
\begin{eqnarray*}
\left\Vert x\xi -\xi x\right\Vert ^{2} &=&\left\Vert x\xi \right\Vert
^{2}+\left\Vert \xi x\right\Vert ^{2}-2\mathrm{Re}\left\langle x\xi ,\xi
x\right\rangle \\
&=&2\left\Vert x\right\Vert _{2}^{2}-2\mathrm{Re}\left\langle \phi
(x),x\right\rangle _{L^{2}(M)} \\
&=&2\mathrm{Re}\left\langle x-\phi (x),x\right\rangle _{L^{2}(M)} \\
&\leq &2\left\Vert x-\phi (x)\right\Vert _{2}\left\Vert x\right\Vert
_{2}\leq \delta \text{.}
\end{eqnarray*}%
By assumption, one can choose $F$ and $\delta \leq \varepsilon $ in such a
way that this guarantees the existence of nonzero projection $p\in Z$ such
that $p\leq p_{0}$ and $\left\Vert x\xi -\xi x\right\Vert _{2}\leq \tau
\left( q\right) ^{1/2}\left\Vert x\right\Vert \varepsilon $ for every
projection $q\in Z$ such that $q\leq p$, and for $x\in qBq$. For such a $%
q\in Z$ and $x\in qBq$ we have that%
\begin{eqnarray*}
\left\Vert \phi (x)-x\right\Vert _{2}^{2} &=&\left\Vert \phi (x)\right\Vert
_{2}^{2}+\left\Vert x\right\Vert _{2}^{2}-2\mathrm{Re}\tau \left( \phi
(x)^{\ast }x\right) \\
&\leq &\left( \tau \circ \phi \right) \left( x^{\ast }x\right) +\tau \left(
x^{\ast }x\right) -2\mathrm{Re}\tau \left( \phi (x)^{\ast }x\right) \\
&=&\left\Vert x\xi -\xi x\right\Vert ^{2}\leq \tau \left( q\right)
\left\Vert x\right\Vert ^{2}\varepsilon ^{2}\text{.}
\end{eqnarray*}

(2)$\Rightarrow $(3) Suppose that $\phi :M\rightarrow M$ is a ucp map such
that $\tau \circ \phi \leq \tau $, $\phi \left( 1\right) \leq 1$, $\phi
=\phi ^{\ast }$, and $\left\Vert \phi \left( 1\right) -1\right\Vert _{2}\leq
\delta $. Since $T_{\phi }^{\ast }=T_{\phi ^{\ast }}=T_{\phi }$, we have
that $\tau \left( x\phi (y)\right) =\tau \left( \phi (x)y\right) $ for $%
x,y\in M$. Consider then the ucp map $\psi :M\rightarrow M$ defined by%
\begin{equation*}
\psi (x)=\phi (x)+\left( 1-\left( \tau \circ \phi \right) \left( 1\right)
\right) \mathrm{E}_{Z}(x)
\end{equation*}%
where $\mathrm{E}_{Z}:M\rightarrow Z\subset M$ is the canonical
trace-preserving conditional expectation. Observe that $T_{\mathrm{E}_{Z}}=%
\mathrm{e}_{Z}:L^{2}(M)\rightarrow L^{2}(Z)\subset L^{2}(M)$ is the
orthogonal projection. Therefore we have that $\tau \left( x\mathrm{E}%
_{Z}(y)\right) =\tau \left( \mathrm{E}_{Z}(x)y\right) $ for every $x,y\in M$%
. Thus $\tau \left( x\psi (y)\right) =\tau \left( \psi (x)y\right) $.
Furthermore we have that $\left( \tau \circ \psi \right) \left( 1\right) =1$%
. From this we deduce that%
\begin{eqnarray*}
\left\Vert \psi \left( 1\right) -1\right\Vert _{2}^{2} &=&\left\Vert \psi
\left( 1\right) \right\Vert _{2}^{2}+1-2\mathrm{Re}\tau \left( \psi \left(
1\right) \right) \\
&=&\tau (\psi \left( 1\right) ^{2})-1\leq \tau \left( \psi \left( 1\right)
\right) -1=0\text{.}
\end{eqnarray*}%
Thus $\psi $ is unital, which implies that $\psi $ is trace-preserving.

Observe now that since $\left\Vert \phi \left( 1\right) -1\right\Vert
_{2}\leq \delta $, we have $\left\vert 1-\left( \tau \circ \phi \right)
\left( 1\right) \right\vert \leq \delta $.\ Therefore for a projection $q\in
Z$ and $b\in qMq$ we have that%
\begin{equation*}
\left\Vert \psi (b)-\phi (b)\right\Vert _{2}\leq \tau \left( q\right)
^{1/2}\delta \left\Vert b\right\Vert \text{.}
\end{equation*}%
This easily gives the desired implication.

(3)$\Rightarrow $(4) Suppose that $\phi :M\rightarrow M$ is a normal ucp $A$%
-bimodule map such that $\tau \circ \phi \leq \tau $, $\phi \left( 1\right)
\leq 1$. Observe then that $\psi :=\frac{1}{2}\left( \phi +\phi ^{\ast
}\right) :M\rightarrow M$ is a normal cp $A$-bimodule map satisfying $\psi
\left( 1\right) \leq 1$, $\tau \circ \psi \leq \tau $, $\psi =\psi ^{\ast }$%
.\ Furthermore for a projection $q\in Z$ and a unitary $u\in qMq$ one has
that%
\begin{equation*}
\left\Vert \psi \left( u\right) -u\right\Vert _{2}\leq \frac{1}{2}\left\Vert
\phi \left( u\right) -u\right\Vert _{2}+\frac{1}{2}\left\Vert \phi ^{\ast
}\left( u\right) -u\right\Vert _{2}\leq 2\left\Vert \phi \left( u\right)
-u\right\Vert _{2}^{\frac{1}{2}}
\end{equation*}%
by \cite[Lemma 1.1.5]{popa_class_2006}. Since every element $x$ of $qMq$
with $\left\Vert x\right\Vert <1$ is a convex combination of unitaries, this
suffices; see also \cite[Lemma 3]{peterson_notion_2005}.

(4)$\Rightarrow $(5) Let $F$ be a finite subset of the unitary group of $M$
containing $1$, and $\delta >0$. Suppose that $\mathfrak{H}$ is an $M$%
-bimodule and $\xi \in M$ is an $A$-central and $\left( F,\delta \right) $%
-central unit vector satisfying $\left\langle \xi ,\cdot \xi \right\rangle
\leq \tau $ and $\left\langle \xi ,\xi \cdot \right\rangle \leq \tau $. We
can assume that $\xi $ is cyclic. Consider the normal cp map $\phi
:M\rightarrow M$ associated with $(\mathfrak{H},\xi )$. This is defined by $%
\phi (x)=L_{\xi }^{\ast }xL_{\xi }$, where $L_{\xi }\left\vert
x\right\rangle =\xi x$. Then we have that $\phi $ is a normal $A$-module cp
map satisfying%
\begin{equation*}
\left( \tau \circ \phi \right) (x)=\left\langle 1|L_{\xi }^{\ast }xL_{\xi
}|1\right\rangle =\left\langle \xi |x\xi \right\rangle \leq \tau (x)\text{.}
\end{equation*}%
Furthermore%
\begin{equation*}
\left\langle x|\phi \left( 1\right) |x\right\rangle =\left\langle x|L_{\xi
}^{\ast }L_{\xi }|x\right\rangle =\left\langle \xi x|\xi x\right\rangle
=\left\langle \xi |\xi xx^{\ast }\right\rangle \leq \tau \left( xx^{\ast
}\right) =\left\Vert x\right\Vert _{2}^{2}=\left\langle x|x\right\rangle
\end{equation*}%
and thus $\phi \left( 1\right) \leq 1$. We have that $\left( \tau \circ \phi
\right) \left( 1\right) =\left\langle \xi |\xi \right\rangle =1$ since $\xi $
is a unit vector. For $u\in F$,%
\begin{eqnarray*}
\left\Vert \phi \left( u\right) -u\right\Vert _{2}^{2} &=&\left\Vert \phi
\left( u\right) \right\Vert _{2}^{2}+1-2\mathrm{Re}\tau \left( \phi \left(
u\right) ^{\ast }u\right) \\
&\leq &2-2\mathrm{Re}\tau \left( \phi \left( u\right) ^{\ast }u\right) \\
&=&\left\Vert u\xi -\xi u\right\Vert ^{2}\leq \delta
\end{eqnarray*}%
By assumption, choosing $F$ large enough and $\delta $ small enough
guarantees that there exists a nonzero projection $p\in Z$ such that $p\leq
p_{0}$ and, for every projection $q\in Z$ such that $q\leq p$ and $x\in qBq$%
, $\left\Vert \phi (x)-x\right\Vert _{2}\leq \tau \left( q\right)
^{1/2}\left\Vert x\right\Vert \varepsilon $. This implies that, for a
unitary $u\in U\left( qBq\right) $,%
\begin{eqnarray*}
\left\Vert u\xi -\xi u\right\Vert ^{2} &=&\left( \tau \circ \phi \right)
\left( q\right) +\tau \left( q\right) -2\mathrm{Re}\tau \left( \phi \left(
u\right) ^{\ast }u\right) \\
&\leq &2\tau \left( q\right) -2\mathrm{Re}\tau \left( \phi \left( u\right)
^{\ast }u\right) \\
&=&2\mathrm{Re}\tau \left( q\left( q-qu\phi \left( u\right) ^{\ast }\right)
\right) \\
&\leq &2\tau \left( q\right) ^{1/2}\left\Vert q-qu\phi \left( u\right)
^{\ast }\right\Vert _{2} \\
&\leq &2\tau \left( q\right) \varepsilon \text{.}
\end{eqnarray*}%
Therefore $\left\Vert u\xi -\xi u\right\Vert \leq \tau \left( p\right)
^{1/2}2\varepsilon ^{1/2}$ for a unitary $u$ in $qBq$. Since every element $%
b\in qBq$ with $\left\Vert b\right\Vert <1$ is a convex combination of
unitaries, this concludes the proof.

(5)$\Rightarrow $(1) Obvious.
\end{proof}

\begin{remark}
Proposition \ref{Proposition:characterize-T-vN} shows that, if $M$ is a von
Neumann algebra and $B\subset M$ is a subalgebra, then $B\subset M$ is rigid
in the sense of \cite[Definition 4.2.1]{popa_class_2006} if and only if $%
\mathbb{C}1\subset \mathbb{C}1\subset B\subset M$ is rigid in the sense of
Definition \ref{Definition:(T)-vN}; see also \cite[Theorem 1]%
{peterson_notion_2005} and \cite[Theorem 3.1]{ioana_amalgamated_2008}.
\end{remark}

\subsection{von Neumann algebra characterization of property (T) for
groupoids}

Now we use the characterization of property (T) for groupoids from Theorem %
\ref{Theorem:T} together with the characterization of rigidity for
inclusions of von Neumann algebras from Proposition \ref%
{Proposition:characterize-T-vN} to give a characterization of property (T)
for groupoids in terms of the corresponding groupoid von Neumann algebra.

\begin{theorem}
\label{Theorem:charaterize-T}Suppose that $G$ is a discrete pmp groupoid,
and $K\leq H\leq G$ are subgroupoid. Assume that $H$ is ergodic. Let $X$ be
the common unit space of $K,H,G$. The following assertions are equivalent:

\begin{enumerate}
\item $K\leq H\leq G$ has property (T);

\item for every $\varepsilon >0$ and nonzero projection $p_{0}\in L^{\infty
}\left( X\right) $ there exist a finite subset $F$ of $L(G)$ and $\delta >0$
such that for every Hilbert $M$-bimodule $\mathfrak{H}$ with an $\left(
F,\delta \right) $-central and $L(K)$-central tracial unit vector $\xi \in 
\mathfrak{H}$ there is a nonzero projection $p\in Z$ and an $H$-central
vector $\eta \in \mathfrak{H}$ such that $p\leq p_{0}$ and $\left\Vert q\eta
-q\xi \right\Vert \leq \tau \left( q\right) ^{1/2}\varepsilon $ for every
projection $q\in L^{\infty }\left( X\right) $ such that $q\leq p$;

\item the inclusion $L^{\infty }\left( X\right) \subset L(K)\subset
L(H)\subset L(G)$ is rigid;

\item for every $\varepsilon >0$ there is a finite subset $F$ of $L(G)$ and $%
\delta >0$ such that for any Hilbert $M$-bimodule $\mathfrak{H}$ with an $%
\left( F,\delta \right) $-central and $L(K)$-central tracial unit vector $%
\xi \in \mathfrak{H}$ there is a nonzero projection $p\in L^{\infty }\left(
X\right) $ such that $\left\Vert x\xi -\xi x\right\Vert \leq \tau \left(
p\right) ^{1/2}\left\Vert x\right\Vert \varepsilon $ for $x\in pL(H)p$.
\end{enumerate}
\end{theorem}

\begin{proof}
(1)$\Rightarrow $(2) Suppose that $K\leq H\leq G$ has property (T). Fix $%
\varepsilon >0$ and a nonzero projection $p_{0}\in L^{\infty }\left(
X\right) $. Then $p_{0}$ can be seen as the characteristic function of some
Borel subset $B$ of $X$. Let $F$ be a finite subset of $\left[ G\right] $
and $\delta >0$ be obtained from $\varepsilon $ and $B$ via Item (7) of the
characterization of property (T) for triples of groupoids provided by
Theorem \ref{Theorem:T}. Let now $\mathfrak{H}$ be an $L(G)$-bimodule with
an $L(K)$-central and $\left( F,\delta \right) $-central tracial unit vector 
$\xi ^{0}\in L$. We can assume that $\xi ^{0}$ is a cyclic vector for $%
\mathfrak{H}$. The assignment $a\mapsto \left( \xi \mapsto a\xi \right) $
defines a normal *-representation of $L^{\infty }\left( X\right) $ on $%
\mathfrak{H}$. Thus there is a Hilbert bundle $\mathcal{H}=(\mathcal{H}%
_{x})_{x\in X}$ such that $\mathfrak{H}=L^{2}\left( X,\mathcal{H}\right) $
and, for $\xi =\left( \xi _{x}\right) _{x\in X}\in L^{2}\left( X,\mathcal{H}%
\right) $ and $a=\left( a_{x}\right) _{x\in X}\in L^{\infty }\left( X\right) 
$, $a\xi =\left( a_{x}\xi _{x}\right) _{x\in X}$; see \cite[Theorem 14.2.1]%
{kadison_fundamentals_1986} or \cite[Proposition F.26]{williams_crossed_2007}%
. Suppose now that $t\in \left[ G\right] $ and $a\in L^{\infty }\left(
X\right) $. Observe that $u_{t}a=\theta _{t}(a)u_{t}$ where $\theta
_{t}(a)=\left( a_{s\left( xt\right) }\right) _{x\in X}\in L^{\infty }\left(
X\right) $. Therefore we have that $u_{t}a\xi u_{t}^{\ast }=\theta
_{t}(a)u_{t}\xi u_{t}^{\ast }$. This shows that the operator $\xi \mapsto
u_{t}\xi u_{t}^{\ast }$ on $L^{2}\left( X,\mathcal{H}\right) $ intertwines
the normal *-representations $a\mapsto \left( \xi \mapsto a\xi \right) $ and 
$a\mapsto \left( \xi \mapsto \theta _{t}(a)\xi \right) $ of $L^{\infty
}\left( X\right) $ on $\mathfrak{H}$. Therefore $\xi \mapsto u_{t}\xi
u_{t}^{\ast }$ is a decomposable operator; see \cite[Theorem 7.10]%
{takesaki_theory_2002}, \cite[Theorem F.21]{williams_crossed_2007}, or \cite[Subsection 2.5]{gardella_representations_2014}. This means that exists a section $x\mapsto
\pi _{xt}\in B\left( H_{s\left( xt\right) },H_{x}\right) $ such that $%
u_{t}\xi u_{t}^{\ast }=\left( \pi _{xt}\xi _{s\left( xt\right) }\right)
_{x\in X}$ for $\xi \in L^{2}\left( X,(\mathcal{H}_{x})_{x\in X}\right) $;
see \cite[Definition 7.9]{takesaki_theory_2002}. By considering such a
decomposition when $t$ varies within a countable dense subgroup of $\left[ G%
\right] $, and by the essential uniqueness of such a decomposition \cite[%
Proposition F.33]{williams_crossed_2007}, one obtains a representation $%
\gamma \mapsto \pi _{\gamma }$ of $G$ on $\mathcal{H}$ such that $u_{t}\xi
u_{t}^{\ast }=\left[ \pi \right] _{t}\xi $ for every $t\in \left[ G\right] $
and $\xi \in L^{2}\left( X,\mathcal{H}\right) $. Since by assumption $\xi
^{0}$ is $L(K)$-central, we have that $\xi ^{0}$ is $K$-invariant.
Furthermore since $\xi ^{0}$ is $\left( F,\delta \right) $-central, we have
that $\xi ^{0}$ is $\left( F,\delta \right) $-invariant for $\pi $.
Therefore by the choice of $F$ and $\delta $, there exists a non-null Borel
subset $A$ of $X$ and an element $\eta $ of $L^{2}\left( X,\mathcal{H}%
\right) $ such that $\eta $ is $H$-invariant and $\left\Vert \eta _{x}-\xi
_{x}^{0}\right\Vert \leq \varepsilon $ for every $x\in A$. The vector $\eta $
together with the characteristic function $p$ of $A$ witness that the
desired conclusion holds.

(2)$\Rightarrow $(3) Suppose that $M$ is a tracial von Neumann algebra, \ $%
Z\subset M$ is a subalgebra, $q\in Z$ is a projection, and $\mathfrak{H}$ is
an $M$-bimodule. Consider a $Z$-central unit vector $\xi \in \mathfrak{H}$
and an $M$-central vector $\eta \in \mathfrak{H}$ such that $\left\Vert q\xi
-q\eta \right\Vert \leq \tau \left( q\right) ^{1/2}\varepsilon $. Then for
every unitary $u$ in $qMq$ one has that%
\begin{equation*}
\left\Vert u\xi -\xi u\right\Vert \leq \left\Vert u\xi -u\eta \right\Vert
+\left\Vert \xi u-\eta u\right\Vert \leq 2\left\Vert q\xi -q\eta \right\Vert 
\text{.}
\end{equation*}%
One can easily prove the implication (2)$\Rightarrow $(3) using this
observation.

(3)$\Rightarrow $(4) This is obvious.

(4)$\Rightarrow $(1) Suppose that (4) holds. We verify that Item (8) of
Theorem \ref{Theorem:T} holds. To this purpose, fix $\varepsilon >0$. We
want to find a finite subset $Q$ of $\left[ G\right] $ and $\delta >0$ such
that, for every normalized $K$-invariant function of positive type $\varphi $
on $G$ such that $\max_{t\in Q}\int_{x\in X}\left\vert \varphi \left(
xt\right) -1\right\vert ^{2}d\mu (x)\leq \delta $, there is a non-null Borel
subset $A$ of $X$ such that $\mathrm{Re}\left( 1-\varphi (\gamma )\right)
\leq \varepsilon $ for a.e.\ $\gamma \in AHA$. Consider a finite subset $F$
of $L(G)$ and $\delta >0$ with the property that, for every trace-preserving
ucp $L(K)$-bimodule map $\phi :L(G)\rightarrow L(G)$ satisfying $\max_{x\in
F}\left\Vert \phi (x)-x\right\Vert _{2}\leq \delta $, there exists a nonzero
projection $p\in L^{\infty }\left( X\right) $ such that $\left\Vert \phi
(x)-x\right\Vert _{2}\leq \tau \left( p\right) ^{1/2}\varepsilon \left\Vert
x\right\Vert $ for $x\in pL(H)p$. By Kaplanski's density theorem \cite[%
I.9.1.3]{blackadar_operator_2006}, we can assume without loss of generality
that there exists a finite subset $Q$ of $\left[ G\right] $ such that $%
F=\left\{ u_{t}:t\in Q\right\} \subset L(G)$. Consider a $K$-invariant
normalized function of positive type $\varphi $ on $G$. Suppose that $%
\max_{t\in Q}\int_{x\in X}\left\vert \varphi \left( xt\right) -1\right\vert
d\mu (x)\leq \delta $. Let $\phi :L(G)\rightarrow L(G)$ be the
trace-preserving ucp $L(K)$-bimodule map associated with $\varphi $ as in
Proposition \ref{Proposition:ucp-groupoid}. Observe that, for every $t\in Q$%
, $\phi \left( u_{t}\right) =\varphi (t)u_{t}$, where $\varphi (t)\in
L^{\infty }\left( X\right) $ is the function $x\mapsto \varphi \left(
xt\right) $. Therefore for $t\in Q$ we have that%
\begin{equation*}
\left\Vert \phi \left( u_{t}\right) -u_{t}\right\Vert _{2}=\left\Vert
\varphi (t)u_{t}-u_{t}\right\Vert _{2}\leq \left\Vert \varphi
(t)-1\right\Vert _{2}\leq \delta \text{.}
\end{equation*}%
Therefore by assumption there exists a nonzero projection $p\in L^{\infty
}\left( X\right) $ such that for every projection $q\in L^{\infty }\left(
X\right) $ such that $q\leq p$ one has that $\left\Vert \phi
(x)-x\right\Vert _{2}\leq \tau \left( q\right) ^{1/2}\varepsilon \left\Vert
x\right\Vert $ for $x\in qL(H)q$. Let now $A\subset X$ be a Borel subset
such that $p$ is the characteristic function of $A$. Then we have that, for $%
\sigma \in \left[ AHA\right] $, $u_{\sigma }\in pL(H)p$, $\tau \left(
p\right) =\mu (A)$, and hence%
\begin{equation*}
\mu (A)\varepsilon ^{2}\geq \left\Vert \phi \left( u_{\sigma }\right)
-u_{\sigma }\right\Vert _{2}^{2}=\left\Vert \varphi (\sigma )-1\right\Vert
_{2}^{2}=\int_{x\in A}\left\vert \varphi \left( x\sigma \right)
-1\right\vert ^{2}d\mu \text{.}
\end{equation*}%
Since this holds for every $\sigma \in \left[ AHA\right] $, we conclude that 
$\left\vert \varphi (\gamma )-1\right\vert \leq \varepsilon $ for a.e.\ $%
\gamma \in AHA$. This shows that Item (8) of the characterization of
property (T) for triples of groupoids holds. This concludes the proof.
\end{proof}

\bibliographystyle{amsplain}
\bibliography{biblio-superr}

\end{document}